\documentclass[11pt,reqno]{amsart}

\usepackage{amsmath, amsthm, amssymb}
\usepackage{amsfonts}
\usepackage[ansinew]{inputenc}
\usepackage[dvips]{epsfig}
\usepackage{graphicx}
\usepackage[english]{babel}
\usepackage{hyperref}

\usepackage{mathtools}
\usepackage{cite}
\usepackage{graphicx}
\usepackage{amscd}
\usepackage{bm}
\usepackage{enumerate}

\usepackage{verbatim}
\usepackage{hyperref}
\usepackage{amstext}
\usepackage{latexsym}
\theoremstyle{plain}
\newtheorem{theorem}{Theorem}[section]
\newtheorem{definition}[theorem]{Definition}
\newtheorem{corollary}[theorem]{Corollary}
\newtheorem{proposition}[theorem]{Proposition}
\newtheorem{lemma}[theorem]{Lemma}
\newtheorem{remark}[theorem]{Remark}

\numberwithin{theorem}{section}
\numberwithin{equation}{section}

\newcommand{\average}{{\mathchoice {\kern1ex\vcenter{\hrule height.4pt
width 6pt depth0pt} \kern-9.7pt} {\kern1ex\vcenter{\hrule
height.4pt width 4.3pt depth0pt} \kern-7pt} {} {} }}

\def\R{\mathbb{R}}
\def\loc{\text{loc}}

\def\div{\text{div}}

\renewcommand{\a }{\alpha }
\renewcommand{\b }{\beta }
\renewcommand{\d}{\delta }

\newcommand{\D }{\Delta }

\newcommand{\e }{\varepsilon }
\newcommand{\g }{\gamma}

\newcommand{\G }{\Gamma}
\renewcommand{\l }{\lambda }

\newcommand{\n }{\nabla }
\newcommand{\vp }{\varphi }

\newcommand{\s }{\sigma }

\newcommand{\z }{\zeta}
\renewcommand{\th }{\theta }

\renewcommand{\o }{\omega }
\renewcommand{\O }{\Omega }

\newcommand{\ov}{\overline}

\newcommand{\be}{\begin{equation}}
\newcommand{\ee}{\end{equation}}

\newcommand{\de}{\partial}

\newcommand{\ti}{\widetilde}

\renewcommand{\k}{\kappa}

 \usepackage{mathrsfs}

\newcommand{\calO }{\mathcal{O}}

\newcommand{\calL }{\mathcal{L}}

\newcommand{\calD }{\mathcal{D}}
\newcommand{\calQ }{\mathcal{Q}}
\newcommand{\calB }{\mathcal{B}}

\newcommand{\N}{\mathbb{N}}

\newcommand{\cL}{{\mathcal L}}

\newcommand{\cR}{{\mathcal R}}

\renewcommand{\epsilon}{\varepsilon}

\newcommand{\Ds}{ (-\D)^s}

\newcommand{\x}{ \xi}

\begin{document}
 

\author[Mouhamed Moustapha Fall]
{Mouhamed Moustapha Fall}
\address{M.M.F.: African Institute for Mathematical Sciences in Senegal, 
KM 2, Route de Joal, B.P. 14 18. Mbour, S\'en\'egal}
\email{mouhamed.m.fall@aims-senegal.org, mouhamed.m.fall@gmail.com}

\author[Xavier Ros-Oton]
{Xavier Ros-Oton}
 \address{X.R.: ICREA, Pg. Lluis Companys 23, 08010 Barcelona, Spain; and Universitat de Barcelona, Departament de Matem\`atiques i Inform\`atica, Gran Via de les Corts Catalanes 585, 08007 Barcelona, Spain.}
\email{xros@ub.edu}



\title[Global Schauder theory for minimizers of the $H^s(\Omega)$ energy]{Global Schauder theory for \\minimizers of the $H^s(\Omega)$ energy}

\begin{abstract}
We study the regularity of minimizers of the functional $\mathcal E(u):= [u]_{H^s(\Omega)}^2 +\int_\Omega fu$.
This corresponds to understanding solutions for the regional fractional Laplacian in $\Omega\subset\R^N$.
More precisely, we are interested on the global (up to the boundary) regularity of solutions, both in the  case of free minimizers in $H^s(\Omega)$ (i.e., Neumann problem), or in the case of Dirichlet condition $u\in H^s_0(\Omega)$ when $s>\frac12$.

Our main result establishes the sharp regularity of solutions in both cases: $u\in C^{2s+\alpha}(\overline\Omega)$ in the Neumann case, and $u/\delta^{2s-1}\in C^{1+\alpha}(\overline\Omega)$ in the Dirichlet case.
Here, $\delta$ is the distance to $\partial\Omega$, and $\alpha<\alpha_s$, with $\alpha_s\in (0,1-s)$ and $2s+\alpha_s>1$.
We also show the optimality of our result: these estimates fail for $\alpha>\alpha_s$, even when $f$ and $\partial\Omega$ are $C^\infty$.
\end{abstract}

\maketitle

\section{Introduction and main results}

Given $s\in(0,1)$ and a bounded domain $\Omega\subset\R^N$, we study the regularity of minimizers of energy functionals of the type
\begin{equation}\label{energy-functional}
\mathcal E(u):= [u]_{H^s(\Omega)}^2 +\int_\Omega fu\, dx,
\end{equation}
where $f\in L^2(\Omega)$, 
\[[u]_{H^s(\Omega)}^2=\frac{c_{N,s}}{4}\int_\Omega\int_\Omega \frac{\big|u(x)-u(y)\big|^2}{|x-y|^{N+2s}}\,dxdy,\]
and $c_{N,s}=\frac{s4^s\G\left( \frac{N}{2}+s\right)}{\pi^\frac{N}{2}\G(1-s)}$.

Notice that, as in the classical case $s=1$, one can study two types of minimizers:

\vspace{2mm}

\noindent $\bullet$\, Free minimizers $u\in H^s(\Omega)$ of \eqref{energy-functional} ---i.e., minimizing among all functions in $H^s(\Omega)$---, which corresponds to a \textbf{Neumann problem}.

\vspace{2mm}

\noindent $\bullet$\, Minimizers with prescribed zero boundary data ---i.e., minimizing among functions $u\in H^s_0(\Omega)$---, which corresponds to a \textbf{Dirichlet problem}.
Notice however that one needs $s>\frac12$ in order to have a trace operator\footnote{When $s\leq\frac12$ we have $H^s_0(\Omega)=H^s(\Omega)$ for any Lipschitz domain $\Omega\subset \R^N$, see e.g.  \cite[Theorem 1.4.2.4]{Grisv}.}, thus \emph{the Dirichlet problem  makes sense only for} $s>\frac12$.

\vspace{2mm}

The aim of this paper is to understand the global (up to the boundary) regularity of solutions in both cases:
\[\textit{If $f$ and $\Omega$ are regular enough, what is the regularity of minimizers $u$?}\]
This is the question that motivates our work.

Minimizers of \eqref{energy-functional} solve the equation $(-\Delta)^s_\Omega u=f$ in $\Omega$ in the weak sense, where 
\begin{equation}\label{regional}
(-\Delta)^s_\Omega u(x) = c_{N,s}{\rm p.v.}\int_\Omega \frac{u(x)-u(y)}{|x-y|^{N+2s}}\,dy
\end{equation}
is the so-called regional fractional Laplacian. Recall that this operator,  subjetcs to Dirichlet or Neumann boundary conditions are    generated by     L\`evy-type processes which are  censored to leave $\O$,  see e.g. \cite{BBC}.

Problems  involving the regional fractional Laplacian  have been studied both from the point of view of Probability \cite{BBC,CK,Guan,Guan-Ma,CKS} and of Analysis \cite{Mou,CGV,AFR,Fall-regional}.
The best known regularity results in this direction are those in \cite{Fall-regional}, where the first author established that for $\O$ sufficiently regular and   $f\in L^\infty(\O)$,   solutions $u$ are $C^{2s-\varepsilon}(\overline\Omega)$ in the Neumann case and that $u/\delta^{2s-1}\in C^{1-\varepsilon}(\overline\Omega)$ for the Dirichlet problem,  for every $\varepsilon>0$.

Our main results here establish for the first time Schauder-type estimates for such problem, which we show to be optimal.
They can be summarized as follows:

\begin{theorem}[Neumann problem]\label{thm-Neu}
Let $s\in(0,1)$, and let $\alpha_s\in(0,1-s)$ be given by Lemma~\ref{analysis-beta}, which satisfies $2s+\alpha_s>1$.
Let $\Omega\subset \R^N$ be any bounded $C^{\max(1,2s+\alpha)}$ domain, and $f\in C^\alpha(\overline\Omega)$ with $\int_{\O}f\,dx=0$.  Let    $\alpha<\alpha_s$ with $\alpha+2s\neq1$.

Let $u$ be the minimizer of \eqref{energy-functional} in $H^s(\Omega)$.
Then, $u\in C^{2s+\alpha}(\overline\Omega)$ and 
\[\|u\|_{C^{2s+\alpha}(\overline\Omega)} \leq C\|f\|_{C^\alpha(\overline\Omega)},\]
with $C$ depending only on $N$, $s$, $\alpha$, and $\Omega$.
In addition, if $\alpha$ is such that $2s+\alpha>1$ then 
\[\partial_\nu u=0\quad \textrm{on}\quad \partial\Omega.\]

Moreover, for $\alpha>\alpha_s$, solutions $u$ are in general {not} $C^{2s+\alpha}(\overline\Omega)$.
\end{theorem}
The critical exponent $\alpha_s>0$ is given by an explicit equation in terms of Gamma functions; see Remark \ref{rem1D} below.
For example, in the simplest case $s=\frac12$, it is the first positive solution of $1+\alpha = \frac{1}{\pi}\tan(\pi \alpha)$.   

Next, we notice that Theorem \ref{thm-Neu} implies that the minimizer $u$ satisfies    $\Ds_\O u\in C^\a(\ov \O)$ and thus it     is in fact a classical solution: $\Ds_{\O}u(x)=f(x)$ for all $x\in \ov\O$.     This will follow from the results in Section \ref{sec4} below.

\begin{remark}
It is natural to wonder whether,  for $s\in (0,\frac{1}{2}]$,  solutions to the Neumann problem are $C^1$ up to the boundary or not and have zero normal derivative   when $f$ and $\Omega$ are regular enough.

Our result answers this question for the first time:  solutions are always $C^1(\overline\Omega)$, for every $s\in (0,1)$.
Furthermore, since $\alpha_s<1-s$, then solutions are in general not $C^2(\overline\Omega)$.

We think it is quite surprising that the answer is the same for all $s\in (0,1)$.

\end{remark}

In the Dirichlet case our result reads as follows:

\begin{theorem}[Dirichlet problem]\label{thm-Dir}
Given $s\in(\frac12,1)$, let $\alpha_s\in(0,1-s)$ be given by Lemma~\ref{analysis-beta}.
Let $\Omega\subset \R^N$ be a $C^{2,\b}$ domain and $f\in C^\alpha(\overline\Omega)$, with $\alpha<\min\{\alpha_s,2s-1,\b\}$ and $\alpha+2s\neq1$.

Let $u$ be the minimizer of \eqref{energy-functional} in $H^s_0(\Omega)$.
Then, $u/\delta^{2s-1}\in C^{1+\alpha}(\overline\Omega)$ and 
\[\|u/\delta^{2s-1}\|_{C^{1+\alpha}(\overline\Omega)} \leq C\|f\|_{C^\alpha(\overline\Omega)},\]
with $C$ depending only on $N$, $s$, $\alpha$, and $\Omega$. Here, $\d\in C^{2,\b}(\ov\O)$ is a positive function in $\O$ that coincides with dist$(\cdot,\de\O)$ near $\de\O$.

In addition, if we denote $\psi:={u}/{\delta^{2s-1}} $, we have
\begin{equation}\label{eq:unn-mc-int}
\partial_\nu\psi=-(N-1)  H_{\partial\Omega} \,\psi\quad \textrm{on}\quad \partial\Omega,
\end{equation}
where $H_{\partial\Omega}$ is the mean curvature of $\partial\Omega$ and $\nu$ is the exterior normal of $\de\O$.

Moreover, for $\alpha>\alpha_s$ the quotient $u/\delta^{ 2s-1}$ does not belong in general to $C^{1+\alpha}(\overline\Omega)$.
\end{theorem}

We point out that \eqref{eq:unn-mc-int} does not hold in the local case $s=1$.
Indeed, in that case we have $\psi=-\partial_\nu u$ and hence $\partial_\nu \psi =- (N-1)H_{\partial\Omega} \,\psi+ f$ on $\partial\Omega$, where we  used $\D u=\de^2_\nu u+(N-1)H_{\de\O}\de_\nu u$ on~$\de\O$.\\

{A difficulty that arises in the proof of Theorem \ref{thm-Dir} in dimension $N\geq 2$ is that $\Ds_\O\delta^{2s-1} $ does not  in general  belong to   $L^\infty(\O)$ even if $\O$ is $C^\infty$.  In fact near $\de\O$, we have that $\Ds_\O\delta^{2s-1}(x)  \asymp  g(x) \log (\d)$, with $g|_{\de\O}=(N-1)H_{\de\O}$}.  This leads to the assumption  $\O$  to be of class $C^{2,\b}$ with $\b>\a$.
To overcome this difficulty  we consider  the  correction $ \delta^{2s-1}+b \d^{2s}$, for some appropriately  chosen function $b$ defined on $\de\O$.  This yields also the identity \eqref{eq:unn-mc-int}.

\begin{remark}[The 1D case]\label{rem1D}
Let us briefly explain what happens in dimension $1$.

Let us assume for example that $\Omega=\R_+$.
Then, we will show that any solution to $(-\Delta)^s_{\Omega} u=0$ in $(0,1)\subset\R$ has an expansion of the form
\begin{equation}\label{1D-expansion}
\qquad \qquad u(x) = c_0 +a_0 x^{2s-1} + a_1 x^{\beta_1}+...+a_k x^{\beta_k}+...\quad \textrm{for}\quad x\approx 0,
\end{equation}
where all $\beta_k$'s can be characterized as the positive solutions of
\[\frac{\Gamma(\beta+1)}{\Gamma(\beta-2s+1)\Gamma(2s)} 
= \frac{1}{\pi}\frac{\sin\big(\pi(\beta-2s)\big)\sin(\pi s)}{\sin\big(\pi(\beta-s)\big)}.\]
By an appropriate analysis of such equation\footnote{Notice that when $s=\frac12$ this equation becomes $\pi \beta=\tan(\pi\beta)$. 
In this case, it is easy to visualize and prove all properties of the exponents $\beta_k$.}, we will show that 
\[\beta_{k+1}>1+\beta_k\quad \textrm{for all}\quad k\geq1,\]
and $\beta_1=2s+\alpha_s$.
In addition, we will see that $\alpha_s\in(0,1-s)$ and $2s+\alpha_s>1$.

In case of zero Neumann boundary conditions on $\partial \Omega$, we must have $a_0=0$ in \eqref{1D-expansion}, while in case of zero Dirichlet conditions we must clearly have $c_0=0$.
This already shows that the $C^{2s+\alpha}$ regularity result from Theorem \ref{thm-Neu} cannot hold for $\alpha>\alpha_s$, as stated above.
\end{remark}

In higher dimensions several difficulties arise, and we will have an expansion of this type with some new terms, both coming from the tangential extra variables and from the curvature of the domain.

We expect the assumptions on $\Omega$ in both Theorems \ref{thm-Neu} and \ref{thm-Dir} to be optimal.

\subsection{Preliminaries and definitions}

Let us next give some important definitions that will be used throughout the paper.

First, recall that given any open set $\Omega\subset\R^N$, the regional fractional Laplacian in $\Omega$ is defined by \eqref{regional}, and its associated bilinear form is given by
\[
 D_\O(u,v)=\frac{c_{N,s}}{2}\int_{\O}\int_{\O}\frac{(u(x)-u(y))(v(x)-v(y))}{|x-y|^{N+2s}}\, dxdy
\]
The natural notion of weak solution is the following.

\begin{definition}
 We say that $u$ is a (weak) solution of 
\be \label{eq:init}
 \Ds_{{\O}}u=f  \quad\textrm{ in\quad$\O\cap B_1 $}  
\ee
with zero \emph{Neumann boundary condition} on $\de\O\cap  B_1$ if $u\in H^s(\O)$ and   
$$
 D_\O(u,\vp)=\int_{\O}f\vp\, dx \qquad \textrm{$\forall\vp \in C^\infty_c(B_1)$.}
$$

On the other hand, when $s>\frac12$, we say that $u$ is a (weak) solution of \eqref{eq:init} with zero \emph{Dirichlet boundary condition} on $\de\O\cap B_1$ if $u\in H^s_0(\O)$  and   
$$
 D_\O(u,\vp)=\int_{\O}f\vp\, dx \qquad \textrm{$\forall\vp \in C^\infty_c(B_1\cap\O )$.}
$$
By approximation, we can of course take any $\varphi\in H^s_0(\Omega\cap B_1)$.
\end{definition}

We notice that being a weak solution is equivalent to being a minimizer of the functional~\eqref{energy-functional}, for both the Neumann and the Dirichlet case.

\subsection{Acknowledgements} 

The first author's work is supported by the Alexander von Humboldt foundation. 
The second author was supported by the European Research Council (ERC) under Grant Agreement No 801867.

\subsection{Organization of the paper}

The paper is organized as follows.
In Section \ref{sec2} we completely characterize all 1D homogeneous solutions, and prove \eqref{1D-expansion}.
In Section \ref{sec3} we use this to prove Liouville theorems in a half-space.
Then, in Section \ref{sec4} we provide some new estimates for the regional fractional Laplacian in domains.
Finally, in Sections \ref{sec5} and \ref{sec6} we prove Theorems \ref{thm-Neu} and \ref{thm-Dir}, respectively.

\section{1D homogeneities}
\label{sec2}

The aim of this section is to classify all possible homogeneities in dimension 1 and, as a consequence, to prove \eqref{1D-expansion}.

Roughly speaking, we want to classify all solutions $u\in H^s_\loc([0,\infty))$ of 
\be \label{eq:first-eq}
\Ds_{\R_+}u=0\quad\textrm{in}\quad \R_+.
\ee
To do so, we will define 
\[v(x):=\left\{\begin{array}{ll}
u(x)-u(0) & \quad\textrm{for}\quad x\geq0 \\
0 & \quad\textrm{for}\quad x\leq0,
\end{array}\right.\] 
and notice that we have\footnote{Here, $(-\Delta)^s$ denotes the fractional Laplacian in $\R$, i.e., \eqref{regional} with $\Omega=\R$.} $\Ds v(x)=\Ds_{\R_+}u(x)+u(x)c_{1,s}\int_{\R_-}|x-y|^{-1-2s}\,dy$.
Then by a direct computation, we find 
\be\label{eq:initi}
\begin{cases}
\Ds v -\displaystyle \frac{a_s}{x^{2s}}\, v=0& \qquad\textrm{ in $\R_+$}\\
v=0&  \qquad\textrm{ in $\R_-$,}
\end{cases}
\ee
where $a_s:=\frac{c_{1,s}}{2s}$.
This allows us to use the Caffarelli-Silvestre extension (see e.g.  \cite{CSilv}).   

Now note that, since our aim is to obtain $C^{2s+\a}$ regularity, we will need to consider solutions to \eqref{eq:initi} that grow like  $|x|^{2s+\a}$ at infinity. 
Thus, the equations in \eqref{eq:initi}  needs to be understood in a generalized sense.\\
 In the following of this paper we   let $\chi_R\in C^\infty_c(B_{2R})$ be   such that $ \chi_R\equiv 1$ in $B_{R}$.
\begin{definition}\label{def-sols-growth-AR}
As in\footnote{This kind of definition was originally introduced in \cite{DSV}. However, the definition in \cite{DSV} requires only pointwise convergence of $f_R$, while here (as in \cite{AR19}) we take uniform convergence.} \cite[Section 3]{AR19}, given a function with polynomial growth 
\[\int_{\R^N}\frac{|u(x)|}{1+|x|^{N+2s+k}}\,dx<\infty,\]
we say that 
\[(-\Delta)^s u\ {\stackrel{\mathclap{k}}{=}}\ f\quad \textrm{in}\quad \Omega\subset \R^N\]
if there exists a family  of polynomials $p_R$ of degree at most $k-1$ and a family of functions $f_R$ such that $(-\Delta)^s (u\chi_{R})=f_R+p_R$ in $\O\cap B_{R/2}$, with $f_R\to f$ uniformly in $\Omega$ as $R\to\infty$.

Thanks to \cite[Lemma 3.3]{AR19}, this is equivalent to saying that there exists an extension $\tilde u(x,y)$ with polynomial growth in $\R^{N+1}_+$, satisfying $\tilde u(x,0)=u(x)$ and ${\rm div}(t^{1-2s}\nabla \tilde u)=0$ in $\{y>0\}$, such that 
\[-\lim_{t\to0} t^{1-2s} \partial_t \tilde u = \overline\kappa_s f \quad \textrm{on}\quad \Omega\cap \{y=0\}.\]

It is easy to check that, if $\int_{\R^N}|D^k u(x)|/(1+|x|^{N+2s})dx<\infty$, and $(-\Delta)^s D^k u=D^k f$ in $\Omega$ for some function $f$, then $(-\Delta)^s u\ {\stackrel{\mathclap{k}}{=}}\ f$ in $\Omega$.
\end{definition}
In the extended variables $(x,t)\in\R\times\R_+$ ---and denoting $v(x,t)$ the extension of $v(x)$---, \eqref{eq:initi} reads as
\be\label{eq:initi2}
\begin{cases}
{\rm div}(t^{1-2s}\nabla v)=0& \qquad\textrm{ in $\{t>0\}$}\\
\displaystyle
-\lim_{t\to 0} t^{1-2s} \partial_t v= \ov \k_s 
\frac{a_s}{x^{2s}}v & \qquad\textrm{ on $\{t=0,\ x>0\}$}\\
v=0&  \qquad\textrm{ on $\{t=0,\ x\leq0\}$,}
\end{cases}
\ee
where  $\ov\k_s=\frac{\Gamma(1-s)}{2^{2s-1}\Gamma(s)}$.

If we look for homogeneous solutions 
\begin{equation}\label{look-homogeneous}
\qquad\qquad v(r\cos\th,r\sin\th)= r^\beta \psi(\theta),
\end{equation}
with  $r>0$ and $\theta\in[0,\pi]$,
we are led to the the following eigenvalue problem for the function $\psi$ in the space $H^{1}\big((0,\pi);  \sin(\th) ^{1-2s}d\theta \big)$:
\begin{align}\label{eq:11}
\begin{cases}
-  \big(  \sin(\th) ^{1-2s} \psi'  \big)' =\l \sin(\th) ^{1-2s} \psi& \qquad\textrm{ for $\theta\in(0,\pi)$}\\
 -\lim_{\th\to 0}\sin(\th) ^{1-2s} \psi'(\th)= \ov \k_s a_{s}   \psi (0)&   \\
  \psi(\pi) = 0,&   
\end{cases}
\end{align}
where 
\[\l = \beta(\beta-2s+1).\]
Problem \eqref{eq:11} possesses a sequence of increasing eigenvalues 
 $$
 \l_0(s)< \l_1(s)\leq,\dots,
 $$
 with corresponding   eigenfunctions $\psi_k\in L^2\big((0,\pi);  \sin(\th)^{1-2s}\big)$,  normalized as
$$
\int_{ 0}^{\pi} \sin(\th)^{1-2s} \psi^2_k(\th)\, d\th=1.
$$
Moreover,  see e.g. \cite{Fall-regional}, \footnote{This follows from the Poincar\'e type inequality: $\l_0(s)=\inf_{\stackrel{\psi\in C^1([0,\pi])}{\psi(\pi)=0}}\frac{\int_0^\pi \sin(\th)^{1-2s}(\psi'(\th))^2\, d\th-\ov \k_s a_{s}  \psi(0)^2}{\int_{ 0}^{\pi} \sin(\th)^{1-2s} \psi^2_k(\th)\, d\th }=0$.}   we have  $\l_0(s)=0$.
Using this, we can prove the following.

\begin{lemma}\label{eq:lem-eigen-dec}
Let $\lambda_k$ and $\psi_k$ as above, and let $\b_k=\frac{2s-1}{2}+\sqrt{\l_k(s)+(\frac{2s-1}{2})^2}$. 
Then, the functions 
\[v_k(r,\theta)= r^{\beta_k} \psi_k(\theta)\]
are homogeneous solutions of \eqref{eq:initi2}.

Moreover, if $w\in H^1\big(B_\rho^+;t^{1-2s} dxdt\big)\cap C(\ov{B_\rho^+})$ is any weak solution of
 \begin{align}\label{eq:eq-w-g}
\begin{cases}
{\rm div}(t^{1-2s} \n w)=0  & \qquad\textrm{ in $B_\rho^+$}\\
-\lim_{t\to 0}t^{1-2s} \de_t w=\ov \k_s  a_{s}  x^{-2s} w & \qquad\textrm{ on $\{t=0,\ 0<x<\rho\}$}\\
w=0 &\qquad\textrm{ on $\{t=0,\ -\rho<x\leq0\}$},
\end{cases}
\end{align}
then 
\be
 w(r\cos\th, r\sin\th) =\sum_{k=0}^\infty \o_k r^{\b_k}\psi_k(\th),
\ee
with  $\o_k\in \R$.
\end{lemma}

\begin{proof}
The first part of the lemma holds by construction of $\lambda_k$ and $\psi_k$.

For the second part, notice that since $\psi_k$ are an orthonormal basis of $L^2\big((0,\pi);\sin(\theta)^{1-2s} d\th\big)$, then we can write
\[w(\rho \cos\theta,\rho\sin\theta) = \sum_{k\geq0} \o_k\rho^{{\beta_k}} \psi_k(\theta)\]
for some constants $a_k\in \R$ (recall that here $\rho$ is a fixed positive number).

But then, the function
\[\tilde w(r\cos\theta,r\sin\theta) := \sum_{k\geq0} \o_k r^{{\beta_k}} \psi_k(\theta)\]
is a solution of \eqref{eq:eq-w-g} that coincides with $w$ on $\partial B_\rho$. Moreover, since $\lambda_0(s)=0$, then $\beta_0=\max( {2s-1},0)$ so that by continuity,  $\o_0=0$ for $s\leq 1/2$.  This implies that  $\tilde w  \in H^1(B_\rho;t^{1-2s}dxdt)$.
Now by uniqueness of solutions, the two functions $\tilde w$ and $w$ must coincide.
\end{proof}

As a consequence, we find:

\begin{corollary}
Let $u\in H^s([0,1])\cap C([0,1)) $ be a function satisfying
\[(-\Delta)^s_{\R_+}u =0\quad \textrm{in}\quad (0,1)\subset \R,\]
in the sense that $\Ds v\stackrel{k}{=}a_sx^{-2s}v$ in $(0,1)$ for some $k\in \N$, with $v:=(u-u(0))1_{\R_+}$.
Then, near $x=0$ we have 
\[u(x) = \o_0x^{\beta_0}+\o_1 x^{\beta_1}+\o_2 x^{\beta_2}+...,\]
where $\o_i\in \R$, and the exponents $\beta_i$ are given by Lemma \ref{eq:lem-eigen-dec}.
\end{corollary}

\begin{proof}
It follows from Lemma \ref{eq:lem-eigen-dec}.
\end{proof}

Our next goal is to find an explicit equation for these exponents $\beta_k$.  We start with  the following
\begin{lemma}\label{frac-Lap-powers}
Let $s\in(0,1)$ and $\beta>-1$.
Then, there exists a unique $\beta$-homogeneous function $v_\beta$ satisfying 
\[
\begin{cases}
{\rm div}(t^{1-2s}\nabla v_\beta)=0& \qquad\textrm{ in $\R^2_+$}\\
v_\beta(x,0)=(x_+)^\beta &  \qquad\textrm{ on $\{t=0\}$,}
\end{cases}\]
Moreover, such function satisfies 
\[-\lim_{t\to 0}t^{1-2s} \de_t v_\beta=\ov\k_s\,\frac{\Gamma(\beta+1)}{\Gamma(\beta-2s+1)}\frac{\sin(\pi(\beta-s))}{\sin(\pi(\beta-2s))} \, (x_+)^{\beta-2s}.\]

Equivalently, in $\R$ we have 
\begin{equation}\label{frac-Lap-1D-powers}
(-\Delta)^s (x_+)^\beta \ 
{\stackrel{\mathclap{k}}{=}}\ 
 \frac{\Gamma(\beta+1)}{\Gamma(\beta-2s+1)}\frac{\sin(\pi(\beta-s))}{\sin(\pi(\beta-2s))} \,(x_+)^{\beta-2s}\quad \textrm{for}\quad x>0,
 \end{equation}
where the equality must be understood in the sense of Definition \ref{def-sols-growth-AR}, and with $k>\beta-2s$.
\end{lemma}

\begin{proof}
We will show\footnote{Equivalently, one could do these computations in the extended variables, see \cite[Appendix A]{RS16b} for a related result.} \eqref{frac-Lap-1D-powers}.
Assume first $\beta\in (-1,2s)$.
Then, at $x=1$ we have:
\[(-\Delta)^s (x_+)^\beta = c_{1,s}\int_{\R}\big(1-(1+y)_+^\beta\big)\frac{dy}{|y|^{1+2s}} 
= c_{1,s}\,\frac{2\pi\Gamma(-2s)}{\Gamma(-\beta)\Gamma(1-2s+\beta)}\frac{\cos(\pi s)\sin(\pi(\beta-s))}{\sin(\pi\beta)\sin(\pi(\beta-2s))},\]
where we used \cite[Proposition 4.3]{DRSV}.

Moreover, by definition of $(-\Delta)^s$, we have 
\[c_{1,s} = \frac{4^s\Gamma(s+{\textstyle\frac12})}{\sqrt{\pi}|\Gamma(-s)|} = \frac{2\Gamma(2s)}{|\Gamma(-s)|\Gamma(s)} = \frac{2s}{\pi}\Gamma(2s)\sin(\pi s),\]
where we used
\[\Gamma(z+{\textstyle\frac12})\Gamma(z)=2^{1-2z}\sqrt{\pi}\,\Gamma(2z),\qquad \Gamma(z+1)=z\Gamma(z),\qquad\textrm{and}\qquad \Gamma(1-z)\Gamma(z)= \frac{\pi}{\sin(\pi z)}.\]
Combining the previous expressions, we find that at $x=1$
\[\begin{split}
(-\Delta)^s (x_+)^\beta & =  \frac{2s}{\pi}\Gamma(2s)\sin(\pi s)\,\frac{2\pi\Gamma(-2s)}{\Gamma(-\beta)\Gamma(1-2s+\beta)}\frac{\cos(\pi s)\sin(\pi(\beta-s))}{\sin(\pi\beta)\sin(\pi(\beta-2s))}\\
&= \frac{\Gamma(\beta+1)}{\Gamma(\beta-2s+1)}\frac{\sin(\pi(\beta-s))}{\sin(\pi(\beta-2s))},
\end{split}\]
where we used again the properties of the $\Gamma$ function and that $\sin(2\pi s)=2\sin(\pi s)\cos(\pi s)$.
Thus, by homogeneity, \eqref{frac-Lap-1D-powers} follows in case $\beta<2s$.

Once we have the result for $\beta\in (-1,2s)$, the general result follows by induction on $\lfloor \beta\rfloor$, by noticing that we can take derivatives in \eqref{frac-Lap-1D-powers}.
\end{proof}

As a consequence of the previous Lemma, we find:

\begin{lemma}\label{reg-Lap-powers}
Let $\beta>-1$, and $v_\beta$ be as in Lemma \ref{frac-Lap-powers}.
Then, $v_\beta$ solves \eqref{eq:initi2} if and only if 
\begin{equation}\label{exponents-beta}\frac{\Gamma(\beta+1)}{\Gamma(\beta-2s+1)\Gamma(2s)}
= \frac{1}{\pi}\frac{\sin(\pi(\beta-2s))\sin(\pi s)}{\sin(\pi(\beta-s))} .
\end{equation}
In particular, the exponents $\beta_k$ are characterized as solutions to this equation.
\end{lemma}

\begin{proof}
The statement is equivalent to \eqref{eq:initi} in the generalized sense.
Thus, we need 
\[(-\Delta)^s(x_+)^\beta = \frac{c_{1,s}}{2s}\,(x_+)^{\beta-2s}.\]
Thanks to the previous Lemma, this is equivalent to 
\[
\frac{\Gamma(\beta+1)}{\Gamma(\beta-2s+1)}\frac{\sin(\pi(\beta-s))}{\sin(\pi(\beta-2s))} 
= \frac{1}{\pi}\Gamma(2s)\sin(\pi s),\]
and the result follows.
\end{proof}

Let us next analyze the equation \eqref{exponents-beta}.

\begin{lemma}\label{analysis-beta}
There exists a sequence $0\leq \beta_0<\beta_1<\beta_2<...\to\infty$ of solutions of \eqref{exponents-beta}.
Moreover,  they satisfy
\[\beta_0 = \max\{2s-1,0\},\]
\[ \beta_k\in (k+\beta_0,k+s)\qquad \textrm{and}\qquad \beta_{k+1}>1+\beta_k \quad\textrm{for}\quad k\geq1.\]
In particular, we have $\beta_1=2s+\alpha_s$, with 
\[\alpha_s\in (0,1-s)\qquad \textrm{and}\qquad 2s+\alpha_s>1\]
for all $s\in (0,1)$.
\end{lemma}

\begin{proof}
Let us first notice that both $\beta=2s-1$ and $\beta=0$ are solutions of \eqref{exponents-beta}.
We let
\[\beta_0 := \max\{2s-1,0\},\]
and look for solutions $\beta>\beta_0$.

We consider the two functions
\[h_1(\beta) = \frac{\Gamma(\beta-2s+1)\Gamma(2s)}{\Gamma(\beta+1)}\]
and
\[h_2(\beta)= \frac{\pi\sin(\pi(\beta-s))}{\sin(\pi(\beta-2s))\sin(\pi s)},\]
and we need to study the intersection of their graphs for $\beta\geq \beta_0$.

Let us first analyse the two functions separately, and then combine the informations in order to see their intersections.

First notice that, by definition of the Beta function and its relation to the Gamma function, for $\beta>2s-1$ we have
\[h_1(\beta) = B(\beta-2s+1,2s) = \int_0^1 t^{\beta-2s}(1-t)^{2s-1}dt.\]
On the other hand, using $\sin(a+b)=\sin a\cos b+\sin b\cos a$ we find
\[h_2(\beta) = \frac{\pi}{\tan(\pi(\beta-2s))} + \frac{\pi}{\tan(\pi s)}.\]

The function $h_1$ satisfies:
\[\lim_{\beta\downarrow 2s-1}h_1(\beta)=+\infty,\qquad h_1(0)=\frac{\pi}{\sin(2\pi s)}, \qquad \lim_{\beta\to+\infty} h_1(\beta)=0,\]
and
\[h_1(\beta)\ \textrm{is positive, decreasing, and convex in}\ (2s-1,\infty).\]
The function $h_2$ satisfies:
\[h_2(\beta)\ \textrm{is 1-periodic},\qquad \lim_{\beta\downarrow 2s-1}h_2(\beta)=+\infty,\qquad h_2(0)=\frac{\pi}{\sin(2\pi s)},\]
\[h_2(\beta)\ \textrm{is decreasing in the intervals }(2s-1,2s)+\mathbb N,\ \,\textrm{convex in }(2s-1,2s-{\textstyle \frac12})+\mathbb N, \]
\[\textrm{it is positive in the intervals }(2s-1,s)+\mathbb N,\ \,\textrm{and negative in }(s,2s)+\mathbb N. \]

We then split the analysis into three cases:

\vspace{2mm}

\noindent\emph{Case 1}. Assume $s=\frac12$.
This is the easiest case, and we have $h_1(\beta)=1/\beta$ and $h_2(\beta)=\pi/\tan(\pi(\beta-1))$.
It is easy to see that these two functions intersect exactly once at each interval $(k,k+1)$, with $k\in\mathbb N$, at a point $\beta_k\in (k,k+\frac12)$ for $k\geq1$ satisfying $\beta_{k+1}>1+\beta_k$.

\vspace{2mm}

\noindent\emph{Case 2}. Assume $s>\frac12$.
In this case, $\beta_0=2s-1$.
Then, we claim that the functions $h_1$ and $h_2$ intersect exactly once in each of the intervals $(2s-1+k,2s+k)$, $k\in\mathbb N$, $k\geq1$.

Indeed, notice first that $h_1$ is continuous, positive, and decreasing in $(2s-1,\infty)$, while $h_2$ is continuous, positive, and decreasing in $(2s-1+k,s+k)$, $k\in \mathbb N$, with $h_2\to+\infty$ as $\beta\downarrow 2s-1+k$ and $h_2(s+k)=0$.
Thus, these two functions intersect at least once in each of the intervals $(2s-1+k,s+k)$, $k\geq1$, and they never intersect in $[s+k,2s+k]$ (since $h_2\leq0$ therein).

It remains to see that they cannot intersect twice in one of these intervals.
For this, we notice that both $h_1$ and $h_2$ are differentiable in the intervals $(2s-1+k,s+k)$. Moreover,  for any $\beta\in (2s-1+k,s+k)$ and $k\geq1$ we have
\[h_1'(\beta) \geq h_1'(2s) = -\int_0^1 |\log t| (1-t)^{2s-1}dt \geq -1,\]
\[h_2'(\beta) = -\frac{\pi^2}{\sin^2(-\pi (\beta-2s))} \leq -\pi^2.\]
This means that $h_2'<h_1'$ in each of these intervals, and therefore they cannot intersect twice in $(2s-1+k,s+k)$, for $k\geq 1$,  thanks to Rolle's theorem.

It only remains to see that the two functions do not intersect in the first interval $(2s-1,s)$.
In that case, however, we are looking for solutions of $(-\Delta)^s_{\R_+}x^\beta=0$ in $\R_+$, with $\beta\in(2s-1,s)$, which we know they do not exist.

\vspace{2mm}

\noindent\emph{Case 3}. Assume $s<\frac12$.
In this case, $\beta_0=0$.
Then, we claim that the functions $h_1$ and $h_2$ intersect exactly once in each of the intervals $(2s-1+k,2s+k)$, $k\in\mathbb N$, $k\geq1$.

Notice first that $h_1$ is continuous, positive, and decreasing in $[0,\infty)$, while $h_2$ is continuous, positive, and decreasing in $(2s-1+k,s+k)$, $k\in \mathbb N$, with $h_2\to+\infty$ as $\beta\downarrow 2s-1+k$ and $h_2(s+k)=0$.
Thus, these two functions intersect at least once in each of the intervals $(2s-1+k,s+k)$, $k\geq1$, and they never intersect in $[s+k,2s+k]$ (since $h_2\leq0$ therein).

To see that they cannot intersect twice in one of these intervals, notice that for $\beta\in (k+s,k+2s)$ the two functions cannot intersect ---since $h_2$ is negative therein---, while for $\beta\in (2s-1+k,k)$, $k\geq1$, they cannot intersect either ---since $h_1<h_1(0)=h_2(0)=h_2(k)<h_2$ therein.

Instead, for $\beta\in (k,k+s)$, $k\geq1$, we have\footnote{This can be seen by using that $t^{1-2s}<1$ and that, since $|\log t|$ is decreasing and $(1-t)^{2s-1}$ is increasing, then  $\int_0^1 |\log t| (1-t)^{2s-1}dt \leq \int_0^1 |\log t|dt \int_0^1(1-t)^{2s-1}dt=\frac{1}{2s}$.}
\[h_1'(\beta) \geq h_1'(1) = -\int_0^1 t^{1-2s}|\log t| (1-t)^{2s-1}dt \geq -\frac{1}{2s}, \]
\[h_2'(\beta) \leq h_2'(1) = -\frac{\pi^2}{\sin^2(\pi(1-2s))} \leq -\frac{1}{s^2}.\]
This means that in such intervals we have $h_1'>h_2'$, and thus the functions cannot intersect twice in $ (k,k+s)$ for all $k\geq1$ by Rolle's theorem.

Notice also that we have showed that $\beta_k\in (k,k+s)$ for all $k\geq1$.
Moreover, since $h_1$ is decreasing and $h_2$ is periodic, then $\beta_{k+1}>1+\beta_k$.

Finally, notice that in the interval $(0,2s)$ we know that there are no solutions to $(-\Delta)^s_{\R_+}x^\beta=0$ in $\R_+$ for $\beta\in(0,2s)$, and we are done.
\end{proof}

\section{Liouville theorems on the Half-space} 
\label{sec3}

 \subsection{Solutions with growth larger than $2s$}
  We denote $\R^N_+=\{x=(x',x_N)\in \R^{N-1}\times \R\,:\ x_N>0\}$ and  $B_r'(z)$  the  ball in $\R^{N-1}$ with radius $r$ centred at $z$, while $B_r^+:=B_r\cap \R^N_+$.
\begin{definition}\footnote{This definition,  is slightly different from Defnition \ref{def-sols-growth-AR}. Here, we ask more regularity on the reminder $\cR_R$ to simplify the proofs.}
Let $U$ be an open subset of $\R^N$ with $0\in \ov U$.
Let $u\in H^s(U)$,  $f\in L^1_{loc}(U)$ and $\a<1$.  We say that $u$ is a  solution to $\Ds_{ {\R^N_+}} u \stackrel{\a}{=}f$ in $U$ if
for all $R\geq 1,$ there exist a constant $c_R\in \R$ and a function  $\cR_R\in  C^{0,1}(B_{R/2})$,  
such that 
\be \label{eq:def-a-sol}
\Ds_{ {\R^N_+}} (\chi_R u)=f+  c_R+ \cR_R  \qquad\textrm{ in  $U\cap B_{R/2}^+$} ,
\ee
with  $\cR_R(0)=0$
and
$
\|\n\cR_R\|_{L^\infty(B_{R/2})}\leq C_0 R^{\a-1}.
$
\end{definition} 
\begin{remark}\label{rem-gen-to-ok}
\begin{enumerate}
\item[(i)]We observe that  if   $\Ds_{ {\R^N_+}} u \stackrel{\a}{=}0$ in $U$  and $|u(x)|\leq C(1+|x|^{2s-\e})$, with $\e>0$, then   $\Ds_{ {\R^N_+}} u {=}0$ in $U$.
\item[(ii)]  If $u$ satisfies \eqref{eq:def-a-sol} and $|u(x)|\leq C_0(1+|x|^{2s+\a})$. Then for any open set $U'\subset U$ with $U'\cap B_{1}^+\not=\emptyset$,   there exists $C=C(N,s, \a,U')$ such that 
\be \label{eq:est-cR}
|c_R|\leq C\left(  C_0  R^\a+ \|f\|_{L ^1(U')} \right).
\ee
Indeed,  consider $\phi\in C^\infty_c( U'\cap B_{1}^+)$ with $\int_{\R^N}\phi\, dx=1$. Then multiply   \eqref{eq:def-a-sol} by $\phi $ and integrate over $\R^N_+$ to get 
\begin{align*}
c_R= \int_{\R^N}\chi_R u \Ds_{\R^N_+}\phi\, dx-\int_{U'} \phi f\, dx-\int_{U'} \cR_R \phi\, dx.
\end{align*}
Since $ |\Ds_{\R^N_+}\phi(x)|\leq C(1+|x|)^{-N-2s}$, we get  \eqref{eq:est-cR}.
\end{enumerate}

\end{remark}

  The following result is  contained in  \cite{Fall-regional}.
 
\begin{lemma}\label{lem:fr-to-reg}
Let   $v\in   H^s(B_2^+)\cap L^\infty(\R^N) $ be a solution to
 \begin{equation}
\label{eq:flat-eq-good-v}
\Ds_{ {\R^N_+}} v=g\quad\text{in}\quad B_{2}^+,
	\end{equation}
with zero  Dirichlet or Neumann boundary conditions on $\de\R^N_+ \cap B_2$.
	Then  
	$$
	\|v\|_{C^{\b}(\ov {B_{1}^+})}\leq C(N,s,\b)( \|v\|_{L^\infty(\R^N)}+ \|g\|_{ L^\infty(\R^N) }   ),
$$
with $\b=2s-1$ if $2s>1$ and $\b\in (0,2s)$ if $2s\leq 1$.
If moreover  $g\in C^{k}(B_2)$,  $k=1,2$, then 
 $$
\|\n _{x'}^k v\|_{L^{\infty } (B_{1}^+)}\leq C(N,s)    \left( \| v\|_{L^\infty(\R^N)}+\|   g\|_{C^k(B_2)}   \right).
 $$
\end{lemma} 

\subsection{Liouville theorem} 

     \begin{lemma}\label{lem:from-ND-1D}
 	Let   $u\in H^s_{loc}(\ov{\R^N_+})$ satisfy,  for $\a<1$,     with $2s+\a\not=1$ and $2s+\a<2$,  
$$
|u(x)|\leq C(1+|x|)^{2s+\a}
$$
and
$$
\Ds_{ {\R^N_+}}  u\stackrel{\a}{=}  0  \qquad\textrm{ in } \R^N_+,
$$
with zero  Dirichlet or Neumann  boundary conditions on $\de \R^N_+$.
Then  $u(x',x_N)= a\cdot x'+b(x_N)$ in the Neumann case, while  in the Dirichlet case $u(x',x_N)=  a\cdot x' x_N^{2s-1}+\o(x_N)$. Here $a\in \R^{N-1}$ is a constant and $\o$ solves     
$$
\Ds_{ {\R_+}} \o \stackrel{\a}{=}    0 \qquad\textrm{ in } \R_+.
 $$
 \end{lemma}
 \begin{proof}
 By definition, we have 
\be\label{eq:us-to-diff} 
\Ds_{ {\R^N_+}} ( \chi_R u)= c_R+   \cR_R  \qquad\textrm{ in  $B_{R/2}^+$}.  
\ee
Next,  let  $v(x)=(\chi_Ru)(Mx )$,  for $M= R/2$.   Then
 $$
\Ds_{ {\R^N_+}}  v=  M^{2s}  c_R +M^{2s}  \cR_R(M\cdot) \qquad\textrm{ in $B_1^+$} .
 $$
 Now by Lemma \ref{lem:fr-to-reg},  
 $$
\| \n _{x'} v\|_{L^{\infty}(B_{1/2})}\leq C(N,s)  \left( \| v\|_{L^\infty(\R^N)}+M^{2s}|c_R|+\| M^{2s}  \n _{x'}\cR_R(Mx)\|_{L^\infty(B_1)} \right).
 $$
  Hence, recalling the definition of $\cR_R$, and Remark \ref{rem-gen-to-ok}$(ii)$, we get
 $$
R\|\n _{x'}( \chi_R u) \|_{L^{\infty}(B_{R/4})}\leq C(N,s)   R^{2s+\a}.
 $$
 Since  $|u(x)\n_{x'} \chi_R(x)|\leq C R^{2s+\a-1}$,  we deduce that, for all $R>1$,
$$
\|\n _{x'}  u \|_{L^{\infty}(B_{R/4})}\leq C(N,s)   R^{2s+\a-1}.
$$
As a consequence 
\be \label{eq:grow-nab-tan}
|\n _{x'}  u (x)|    \leq C(N,s)  (1+ |x|^{2s+\a-1}) \qquad\textrm{ for all $x\in \R^N_+$.}
\ee
Differentiating \eqref{eq:us-to-diff} in $x'$ and using Lemma \ref{lem:fr-to-reg},  we obtain   $\n_{x' }u\in H^s_{loc}(\ov{\R^N_+})$ and 
\be\label{eq:us-to-diff-tan} 
\Ds_{ {\R^N_+}}  (\chi_R \n_{x' }u)= -\Ds_{ {\R^N_+}}( u \n_{x'}\chi_R)+  \n_{x'} \cR_R  \qquad\textrm{ in  $B_{R/2}^+$.} 
\ee
Thanks to \eqref{eq:grow-nab-tan} and the fact that $ | \n_{x'} \cR_R|+|\Ds_{ {\R^N_+}}( u \n_{x'}\chi_R)|\leq C R^{\a-1}$ in  $B_{R/2}^+$, we can send $R\to \infty$ to deduce that
 $$
 \Ds_{ {\R^N_+}}    \n_{x' }u= 0 \qquad\textrm{ in  $\R^N_+$.}
 $$
 with  $\n_{x' }u \in H^s_{loc}(\ov{\R^N_+})$ satisfies the same boundary condition as $u$. 
In view of \eqref{eq:grow-nab-tan} and since $2s+\a-1<2s$,   by applying the Liouville theorems in \cite{Fall-regional}, we deduce that   $\n_{x'}u(x',x_N)$ is a constant vector on $\R^N_+$ in the Neumann case and is proportional  to $x_N^{\max(2s-1,0)}$ in the Dirichlet case.    
 \end{proof}
 
\begin{lemma}[1D Liouville theorem]\label{eq:1D-L-th}
Let $s\in(0,1)$, and let $\alpha_s$ be given by Lemma~\ref{analysis-beta}.

Let   $\o\in H^s_{loc}(\ov{\R_+})$ satisfy  
\be \label{eq:growth-omega}
\qquad\qquad |\o(x)|\leq C(1+|x|)^{2s+\a},\qquad \textrm{with}\quad \alpha<\alpha_s,
\ee
and
$$
\Ds_{ {\R_+}}   \o \stackrel{\a}{=}    0 \qquad\textrm{ in } \R_+,
 $$
 with zero Dirichlet or zero Neumann at $\de \R_+$. 
 Then, $\o(x)= c_1+c_2 x^{2s-1}$, for some $c_1,c_2\in \R$.
 \end{lemma}
 
 \begin{proof}
 By Lemma \ref{lem:fr-to-reg}, $(\o-\o(0))1_{\R_+}\in H^{s}_{loc}([0,\infty))\cap C ([0,\infty))$.  In addition, we have that
  $$
L_s\left[(\o-\o(0))1_{\R_+}\right]\stackrel{\a}{=}0  \qquad\textrm{ in $\R_+$},
 $$
 where 
$$
 L_s:=\Ds-a_s x^{-2s}. 
$$
By \cite[Lemma 3.3]{AR19}, there exists   $W\in H^1_{loc}(\ov{\R^{2}_+};t^{1-2s}dxdt) \cap C(\ov{\R_2^+})$   such that 
\be \label{eq:polynom-growth}
 |W(x,t)|\leq C (1+(t+|x|)^{K}),
 \ee
 for some constants  $K\in \N$, $C>0$, and 
 \begin{align*}
\begin{cases}
\div(t^{1-2s} \n W)=0& \qquad \textrm{in }\R\times (0, \infty)\\
-t^{1-2s} \de_t W= \ov\k_s a_s x^{-2s} W    & \qquad \textrm{on $\R_+$}\times \lbrace 0 \rbrace\\
W=(\o-\o(0))1_{\R_+} & \qquad \textrm{on } \R \times \lbrace 0 \rbrace.
\end{cases}
\end{align*}  
Thanks to Lemma \ref{eq:lem-eigen-dec},  for all $r>0$, 
$$
W(r\cos\th,  r\sin\th)=\sum_{k=0}^\infty b_k r^{\b_k}\psi_k(\th).
$$ 
By the Parseval identity and \eqref{eq:polynom-growth}, we get 
\begin{align*}
\sum_{k=0}^\infty b_k^2 r^{2\b_k}=\int_{0}^\pi W^2(r\cos\th,  r\sin\th)\sin(\th)^{1-2s}\, d\th\leq C (1+r)^{2K}.
\end{align*}
Hence, letting $r$ to $\infty$ we find that $W(r\cos\th,  r\sin\th)=\sum_{\b_k\leq K} b_k r^{\b_k}\psi_k(\th)$. As a consequence
$$
W(r,  0)=\sum_{\b_k\leq K} b_k r^{\b_k}\psi_k(0).
$$ 
Recall that $2s+\a<\b_1<\b_k $ for all $k>1$. Therefore since  $W(r,  0)=\o(r)-\o(0)$ we then get from \eqref{eq:growth-omega} and the above identity that  $b_k=0$ for all $k>1$, so that 
$\o(r)=\o(0)+ b_0 r^{\b_0}\psi_0(0)$. 
 
 \end{proof}

 Combining Lemma \ref{eq:1D-L-th},  Lemma \ref{lem:from-ND-1D}, and the regularity theory in \cite{Fall-regional, AFR}, where have that $u\in C^{2s-1+\e}(\ov{B_1^+})$ in the zero Neuman case,  we get the

   \begin{theorem}[Liouville theorem: Neumann and Dirichlet]\label{th:Liouville-Neum}
       Let $s\in(0,1)$ and let $\alpha_s\in(0,1-s)$ be given by Lemma~\ref{analysis-beta}.   
 	Let   $u\in H^s_{loc}(\ov{\R^N_+})$ satisfy, 
$$
\qquad \qquad |u(x)|\leq C_0(1+|x|)^{2s+\a},\qquad \textrm{with}\quad \alpha<\alpha_s,
$$
and
$$
\Ds_{{\R^N_+}} u\stackrel{\a}{=}  0.
$$
\begin{enumerate}
\item[(i)] If $u$ satisfies zero Neumann boundary condition on $\de\R^N_+$,  then  $u(x',x_N)=B+A\cdot x'$ for all $x\in \R^N_+$,  for some constants $B\in \R$ and $A\in \R^{N-1}$. 
\item[(ii)] If $s>\frac12$ and $u=0$  on $\de\R^N_+$ then  $u(x)=(B+ A\cdot x') x_N^{2s-1}$ for all $x=(x',x_N)\in \R^N_+$, for some constants $B\in \R$ and $A\in \R^{N-1}$. 
\end{enumerate}
 \end{theorem}

 \section{Some estimates and formulas related to the regional fractional laplacian}
\label{sec4} 
 
In this Section we establish some new estimates related to the regional fractional Laplacian.  To alleviate the notations, we assume in this section that $c_{N,s}$, the constant in \eqref{regional},  is equal to 1.
 We start with  the following:
 
\begin{lemma}\label{lem:mapping-prop}
Let $s\in(0,1)$, $\alpha\in(0,1)$ be such that $1<2s+\alpha<2$, $\Omega\subset\R^N$ be a $C^{2s+\alpha}$ domain, and $u\in C^{2s+\alpha}(\overline\Omega)$ satisfying $\partial_\nu u=0$ on $\partial\Omega$.
Then,
\[\|(-\Delta)^s_\Omega u\|_{C^\alpha(\overline\Omega)} \leq C\|u\|_{C^{2s+\alpha}(\overline\Omega)},\]
with $C$ depending only on $N$, $s$, $\alpha$, and $\Omega$.
\end{lemma}

\begin{proof}
It suffices to see that, for any $x'\in B_r(x)$, where $B_{2r}(x)\subset\Omega$, we have
\[\big|(-\Delta)^s_\Omega u(x)-(-\Delta)^s_\Omega u(x')\big|\leq Cr^\alpha;\]
see e.g. \cite[Appendix A]{FR-book}.

For this, we write
\[\begin{split}
(-\Delta)^s_\Omega u(x') & = \int_{(-x'+\Omega)\cap B_r}\frac{u(x')-u(x'+z)}{|z|^{N+2s}}\,dz + \int_{(-x'+\Omega)\setminus B_r}\frac{u(x')-u(x'+z)}{|z|^{N+2s}}\,dz \\
&=: I_1(x')+I_2(x'),
\end{split} \]
for any $x'\in B_r(x)$.
Notice that $r>0$ is fixed, and that $B_r(x')\subset\Omega$ for all such $x'$.
Let us now treat the two terms $I_1$ and $I_2$ separately.

For $I_1$ we have
\[I_1(x')=\int_{B_r}\frac{2u(x')-u(x'+z)-u(x'-z)}{|z|^{N+2s}}\,dz = \int_{B_r} \int_0^1 [\nabla u(x'+tz)-\nabla u(x'-tz)]\cdot z\,\frac{dt\,dz}{|z|^{N+2s}}.\]
Thus, since 
\[\big|\nabla u(x+tz)-\nabla u(x-tz)-\nabla u(x'+tz)+\nabla u(x'-tz)\big|\leq C|z|^{2s+\alpha-1},\]
we deduce that
\[\big|I_1(x)-I_1(x')\big| \leq C\int_{B_r}|z|^{2s+\alpha-1}\frac{|z|dz}{|z|^{N+2s}} \leq Cr^{\alpha}.\]

For $I_2$ we have
\[\nabla I_2(x') = \int_{\Omega\setminus B_r(x')} \frac{\nabla u(x')-\nabla u(y)}{|x'-y|^{N+2s}}dy +
\int_{\partial\Omega}\frac{u(x')-u(y)}{|x'-y|^{N+2s}}\, \nu(y)\,d\sigma(y).\]
The first term can be bounded as follows.
\[\left|\int_{\Omega\setminus B_r(x')} \frac{\nabla u(x')-\nabla u(y)}{|x'-y|^{N+2s}}dy\right|\leq C\int_{\Omega\setminus B_r(x')} \frac{|x'-y|^{2s+\alpha-1}}{|x'-y|^{N+2s}}dy \leq Cr^{\alpha-1}.\]

For the second term in $I_2$ we have notice that, since $\partial_\nu u=0$ on $\partial\Omega$, then
\[\big|u(x')-u(x^*)\big|\leq Cd^{2s+\alpha},\]
where $x^*\in\partial\Omega$ is the projection of $x'$ onto $\partial\Omega$, and $d={\rm dist}(x',\partial\Omega)\geq r$.
Then
\[|u(x')-u(y)|\leq Cd^{2s+\alpha} + |u(x^*)-u(y)|,\]
and therefore 
\[\left|\int_{\partial\Omega}\frac{u(x')-u(y)}{|x'-y|^{N+2s}}\, \nu(y)\,d\sigma(y)\right| \leq Cd^{\alpha-1} + \left|\int_{\partial\Omega}\frac{u(x^*)-u(y)}{|x'-y|^{N+2s}}\, \nu(y)\,d\sigma(y)\right|.
\]
Using that $\partial\Omega$ is locally a $C^{2s+\alpha}$ graph, we will show that 
\begin{equation}\label{C2s+alpha-domains}
\left|\int_{\partial\Omega}\frac{u(x^*)-u(y)}{|x'-y|^{N+2s}}\, \nu(y)\,d\sigma(y)\right|\leq Cd^{\alpha-1},
\end{equation}
and hence, combining the previous inequalities and using that $d>r$, we find
\[|\nabla I_2(x')|\leq Cr^{\alpha-1}.\]
Thus, since this holds for any $x'\in B_r(x)$, we deduce
\[\big|I_2(x)-I_2(x')\big| \leq C|x-x'| r^{\alpha-1} \leq Cr^\alpha,\]
and the result follows.

It only remains to prove \eqref{C2s+alpha-domains}.
Clearly, it is enough to show 
\begin{equation}\label{C2s+alpha-domains-B}
\left|\int_{\partial\Omega\cap B_{\rho_\circ}(x^*)}\frac{u(x^*)-u(y)}{|x'-y|^{N+2s}}\, \nu(y)\,d\sigma(y)\right|\leq Cd^{\alpha-1},
\end{equation}
for some small $\rho_\circ>0$.
For this, notice that 
\[u(y) = u(x^*) + \nabla u(x^*)\cdot (y-x^*) + O\big(|y-x^*|^{2s+\alpha}\big)\]
and
\[\nu(y) = \nu(x^*) + O\big(|y-x^*|^{2s+\alpha-1}\big).\]
Therefore, 
\[\big(u(y)-u(x^*)\big)\nu(y) = \nabla u(x^*)\cdot (y-x^*)\nu(x^*) + O\big(|y-x^*|^{2s+\alpha}\big).\]
Since $\nu(x^*)$ and $\nabla u(x^*)$ are fixed (and bounded), then
\[\left|\int_{\partial\Omega\cap B_{\rho_\circ}(x^*)} \frac{u(x^*)-u(y)}{|x'-y|^{N+2s}}\, \nu(y)\,d\sigma(y)\right|  \leq  C\left|\int_{\partial\Omega\cap B_{\rho_\circ}(x^*)} (y-x^*)\,\frac{d\sigma(y)}{|x'-y|^{N+2s}}\right| +  Cd^{\alpha-1}.
\]
To bound the last integral, let us assume that $x^*=0$ and that $\partial\Omega \cap B_{\rho_\circ}$ can be written as a graph in the $e_N=-\nu(x^*)$ direction,  i.e., $\partial\Omega \cap B_{\rho_\circ}(x^*) = \{x_N = g(x_1,...,x_{N-1})\}$.
We denote $z=(x_1,...,x_{N-1})$, and then
\[\int_{\partial\Omega\cap B_{\rho_\circ}} (y-x^*)\,\frac{d\sigma(y)}{|x'-y|^{N+2s}}
= \int_{B_{\rho_\circ}'} \big(z,g(z)\big)\,\frac{J(z)dz}{\big(|d-g(z)|^2+|z|^2\big)^{\frac{N+2s}{2}}},\]
where $J(z)$ is the Jacobian of the change of variables, and $B_{\rho_\circ}'\subset \R^{N-1}$.

Since $g\in C^{2s+\alpha}$ and $g(0)=|\n g(0)|=0$, then $|g(z)|\leq C|z|^{2s+\alpha}$, and the last component of such integral can be bounded by
\[\int_{B_{\rho_\circ}'} |g(z)|\,\frac{J(z)dz}{\big(|d-g(z)|^2+|z|^2\big)^{\frac{N+2s}{2}}} \leq C\int_{B_{\rho_\circ}'} |z|^{2s+\alpha}\,\frac{dz}{\big(d^2+|z|^2\big)^{\frac{N+2s}{2}}} \leq Cd^{\alpha-1},\]
where we used\footnote{It is easy to see that $|a-b|^2\geq \frac12a^2-b^2$ for all $a,b\in\R$.} that $|d-g(z)|^2+|z|^2 \geq \frac12d^2-|g(z)|^2+|z|^2 \geq \frac12(d^2+|z|^2)$ in $B_{\rho_\circ}$.

To control the first $n-1$ components, we symmetrize the integral, and then we need to bound
\[\frac12\int_{B_{\rho_\circ}'} z\left\{\frac{J(z)}{\big(|d-g(z)|^2+|z|^2\big)^{\frac{N+2s}{2}}} - \frac{J(-z)}{\big(|d-g(-z)|^2+|z|^2\big)^{\frac{N+2s}{2}}} \right\}dz.\]
As before, we have $|d-g(\pm z)|^2+|z|^2 \asymp d^2+|z|^2$, and hence
\[\left|\big(|d-g(z)|^2+|z|^2\big)^{-\frac{N+2s}{2}} - \big(|d-g(-z)|^2+|z|^2\big)^{-\frac{N+2s}{2}}\right| \leq C\big|g(z)-g(-z)\big|\big(d^2+|z|^2\big)^{-\frac{N+2s+1}{2}}.\]
Moreover, 
\[\big|J(z)-J(-z)\big|\leq C|z|^{2s+\alpha-1}\qquad\textrm{and}\qquad \big|g(z)-g(-z)\big|\leq C|z|^{2s+\alpha}.\]
Therefore, 
\[\begin{split}
\left|\int_{B_{\rho_\circ}'} z\,\frac{J(z)dz}{\big(|d-g(z)|^2+|z|^2\big)^{\frac{N+2s}{2}}} \right| & \leq C\int_{B_{\rho_\circ}'} |z|\left\{\frac{|z|^{2s+\alpha}}{\big(d^2+|z|^2\big)^{\frac{N+2s+1}{2}}} + \frac{|z|^{2s+\alpha-1}}{\big(d^2+|z|^2\big)^{\frac{N+2s}{2}}} \right\}dz\\ 
& \leq C\int_{B_{\rho_\circ}'} \frac{|z|^{2s+\alpha}}{\big(d^2+|z|^2\big)^{\frac{N+2s}{2}}} \,dz \leq Cd^{\alpha-1}.
\end{split}\]
Hence, \eqref{C2s+alpha-domains-B} follows, and we are done.
\end{proof}

In case $2s+\alpha<1$ we have the following.
 
 \begin{lemma}\label{lem:2s-a-s-1}
Let $s,\alpha>0$ be such that $2s+\alpha<1$, let $\Omega\subset\R^N$ be any open set, and let $u\in C^{2s+\alpha}(\overline\Omega)$.
Then,
\[\|(-\Delta)^s_\Omega u\|_{C^\alpha(\overline\Omega)} \leq C\|u\|_{C^{2s+\alpha}(\overline\Omega)},\]
with $C$ depending only on $N$, $s$, $\O$ and $\alpha$.
 \end{lemma}
 \begin{proof}
 It is easy to see that  
$$
 \|\Ds_{\O} u\|_{L^{\infty}(\O)}\leq \sup_{x\in \O}\int_{\O}\frac{|u(x)-u(y)|}{|x-y|^{N+2s}}\, dy\leq  C\|u\|_{C^{2s+\a}(\ov \O) } .
$$
 Now, for $x,x'\in \O$ let  $r=2{|x-x'|}$.  We have 
 \begin{align*}
&J(x,x'):= \Ds_{\O} u(x)- \Ds_{\O} u(x')=\int_{|x-y|<r}\frac{u(x)-u(y)}{|x-y|^{N+2s}}\, dy- \int_{|x-y|<r}\frac{u(x')-u(y)}{|x'-y|^{N+2s}}\, dy\\
 &+(u(x)-u(x'))\int_{|x-y|>r}\frac{1}{|x'-y|^{N+2s}}\, dy\\
 &- \int_{|x-y|>r}(u(x)-u(y))\left(\frac{1}{|x-y|^{N+2s}}  - \frac{1}{|x'-y|^{N+2s}} \right)\, dy.
 \end{align*}
 We observe that if $|y-x|<r$, then $|y-x'|\leq 3r/2$ and  if $|y-x|>r$, then $|y-x'|\geq r/2$.  Moreover  for all $t\in (0,1)$, if $|x-y|\geq r$, then $|tx+(1-t)x'-y|\geq |y-x|/2$.
 We then have 
 \begin{align*}
 &|J(x,x')|\leq C r^{\a}\\
 &+ C r^{2s+\a}\int_{|x'-y|>r/2} |x'-y|^{-N-2s}\, dy+ C \int_0^1\int_{|x-y|\geq r}\frac{|x-x'|  |x-y|^{2s+\a}dy}{ |tx+(1-t)x'-y|^{N+2s+1}}dt\\
 &\leq C r^{\a}+C |x-x'|   \int_{|x-y|\geq r}|x-y|^{2s+\a} |x-y|^{-N-2s-1}\leq C r^\a.
 \end{align*}
 The proof is complete.
 \end{proof}


Finally, we provide a new formula for the regional fractional Laplacian on domains.
This can be used to give a different proof of Lemma \ref{lem:mapping-prop} in case $s>\frac12$.

 \begin{lemma}\label{lem:rew-reg-op}
 Let $\O$ be a bounded $C^{\max(2s,1)}$ domain.  
 Let $u\in C^{2s+\a}(\ov{\O})$, with $2s+\a>1$ and $\de_{\nu} u=0$ on $\de\O$.  
 Then,  for  $x\in \ov\O$ we have
 \begin{align*}
\Ds_{\O}u(x)&=\int_{{\O}}\frac{u(x)-u(y)- \n u (x)\cdot (x-y) }{|x-y|^{N+2s}}\, dy \,+\\
&\hspace{45mm}+\frac{1}{N+2s-2}\int_{\de\O}\frac{\big(\n u(x)-\n u(y) \big)\cdot\nu(y)}{|x-y|^{N+2s-2}}\,d\s(y).
 \end{align*}
 where $\nu$ is the unit exterior normal of $\O$. Moreover, all the above integrals converge absolutely.
 \end{lemma}

Notice that, for simplicity, we establish the result in case of bounded domains.
Still, an analogous result could be proved for unbounded domains.

 \begin{proof}[Proof of Lemma \ref{lem:rew-reg-op}]
 We have 
  \begin{align*}
\Ds_{\O}u(x)&=\int_{{\O}}\frac{u(x)-u(y)- \n u (x)\cdot (x-y) }{|x-y|^{N+2s}}\, dy +p.v.\int_{{\O }}\frac{   \n u (x)\cdot(x-y) }{|x-y|^{N+2s}}\, dy.
 \end{align*}
  Define 
  $$
  I_\e(x):= \int_{\O\cap \de B_\e(x)}\frac{\n u(x)\cdot \eta_\e(y)}{|x-y|^{N+2s-2}}d\s(y),
  $$
  where $\eta_\e$ is the unit exterior normal of $B_\e(x)$.
 Using that $\n_y  |x-y|^{-N-2s+2}= (N+2s-2) (x-y) |x-y|^{-N-2s}$,  the divergence theorem and the fact that $\n u (y)\cdot \nu(y)=0$ for $y\in \de\O$,  we then get 
 \begin{align*}
(N+2s-2)p.v. \int_{{\O}}\frac{   \n u (x)\cdot(x-y) }{|x-y|^{N+2s}}\, dy &=  p.v.\int_{\de\O}\frac{\n u(x)\cdot\nu(y)}{|x-y|^{N+2s-2}}d\s(y)\\
&=   \int_{\de\O}\frac{[\n u(x)-\n u(y)]\cdot\nu(y)}{|x-y|^{N+2s-2}}d\s(y)+ \lim_{\e\to 0}I_\e(x).
 \end{align*}
 If $x\in \O$ then we immediately get $ \lim_{\e\to 0}I_\e(x)=0$ and the result follows.  We now consider the case     $x\in \de\O$,  then denote by $H_x:=T_x\de\O$ the tangent plane of $\de\O$ at $x$ and $H^+_x$ the upper half space containing $-\nu(x)$ with boundary $H_x$.  
The proof is now complete once we show that
 \be\label{eq:pv-remainder}
I_\e(x)=-C(N)\e^{1-2s} \n u(x)\cdot \nu(x)+o_\e(1)=o_\e(1),
 \ee
 where $C(N)=\int_{S^{N-1}_+}\th_N\, d\s(\th)$.
We  have 
\begin{align}
&I_\e(x)=\int_{\O\cap \de B_\e(x)}\frac{\n u(x)\cdot \eta_\e(y)}{|y-x|^{N+2s-2}}d\s(y)=\e^{-N-2s+1} \int_{\O\cap \de B_\e(x)}{\n u(x)\cdot (y-x)} d\s(y) \nonumber\\
&=\e^{-N-2s+1} \int_{H_x^+\cap \de B_\e(x) }{\n u(x)\cdot (y-x)} d\s(y)-\e^{-N-2s+1} \int_{\O\bigtriangleup H_x^+\cap \de B_\e(x)} {\n u(x)\cdot (y-x)}  d\s(y).
\label{eq:I-epsi-normal}
\end{align}
By oddness, we have  $\int_{H_x^+\cap \de B_\e(x) }{\n u(x)\cdot (y-x)} d\s(y)=-\e^{N}\n u(x)\cdot \nu(x)\int_{S^{N-1}_+}\th_N\, d\s(\th)$. On the other hand, since $\O$ is of class $C^{2s}$,  we can write a neighbourhood of $x$ in $\de\O$  as a graph of a function $y_N=\g(y')=o(|y'|^{2s} )$ over the tangent plane $H_x$.  We then observe  that  if $y\in \O \bigtriangleup H_x^+\cap \de B_\e(x)$ then $|y'-x|\leq\e$ and $|y_N|\leq \g(y')$, so that   
$$
|\O \bigtriangleup H_x^+ \cap \de B_\e(x)|=\e^{N-2} o(\e^{2s}) .
$$
We thus conclude from this and \eqref{eq:I-epsi-normal} that \eqref{eq:pv-remainder} holds.
 \end{proof}

 \begin{remark}
 It follows from the above proof that the map
 $$
 x\mapsto\int_{\O\cap |x-y|>\e}\frac{u(x)-u(y)}{|x-y|^{N+2s}}\, dy
 $$
 converges in $C_{loc}(\ov\O)$ to $\Ds_\O u$, as $\e\to 0$. 
 The proof also shows that for  $\Ds_\O u(x)$ to be defined at $x\in \de\O$, then necessarily, $\de_\nu u(x)=0$. See also \cite{Guan-Ma}, where these type of results were obtained.   
 \end{remark}

 \section{Regional Neumann problem}
\label{sec5}

In this section, we consider $\O$ a domain of class $C^{\max(1, 2s+\a)}$ with $0\in \de\O$ and $\nu(0)=-e_N$.
 For,  $2>2s+\a>1$,  we suppose that there exists a global  diffeomorphism $\psi\in C^{2s+\a}(\R^{N};\R^N)$,   satisfying 
 \be\label{eq: diffeom}
 \de\O\cap B_2  =\psi( B_2'\times \{0\}),  \qquad  \O\cap B_2=\psi(  B_2^+),    \qquad   \psi(0)=0
 \ee
 and
 \be  \label{eq:de-n-psi-normal}
-\frac{ \de_{y_N}\psi(y',0)}{|\de_{y_N}\psi(y',0)| }=\nu(\psi(y',0)) \quad\textrm{ for $y'\in  B_2'$},
 \ee
 where $\nu$ is the exterior normal of $ \de\O $,  see \cite{AR19} for the construction of such map. \\
 In the case $2s+\a<1$,   we simply assume that  $\de\O\cap B_2$ is parameterized by a global diffeomorphism  $\psi\in C^1(\R^N;\R^N)$   satisfying \eqref{eq: diffeom}.\\
Next, we define the operator  $\cL_\psi $ by
\be \label{eq:cL-psi}
\cL_\psi w(x):=c_{N,s}\int_{B_2^+}\frac{w(x)-w(y)}{|\psi(x)-\psi(y)|^{N+2s}}Jac_\psi(y)\, dy,
\ee
so that for $u(x)=w(\psi(x))$, we get
\be\label{eq:change-of-var}
\Ds_{\psi(B_2^+)}u(\psi(x))=\cL_\psi w(x). 
\ee
We also define the bilinear form $\calD_\psi: H^s(\R^N_+)\times H^s(\R^N_+)\to \R$ corresponding to this operator as 
$$
  \calD_\psi(w,\vp)=\frac{c_{N,s}}{2}\int_{B_2^+\times B_2^+}\frac{(w(x)-w(y))(\vp(x)-\vp(y))}{|\psi(x)-\psi(y)|^{N+2s}} Jac_\psi(x)Jac_\psi(y)\, dy dx.
$$
 Let  $A\in \R^{N-1} =\{y_N=0 \}$ and define 
 \be \label{eq:q_A}
 q_A(y)=A\cdot y =A\cdot y'  
 \ee
 We now define 
$$
Q_A: \R^N\to \R^N, \qquad Q_A(x)=q_A(\psi^{-1}(x))  .
$$
 We have 
 $$
 Q_A(\psi(y))=q_A(y)=A\cdot y',
 $$
and thus
 $$
0 =\de_{y_N}q_A(y)= \n Q_A(\psi(y))\cdot   \de_{y_N}\psi(y).
 $$
 This, together  with \eqref{eq:de-n-psi-normal}, imply that 
$$
 \de_\nu Q_A  =0 \qquad\textrm{ on  $ \de\O \cap B_2 $.}
$$
Hence,  since $Q_A\in C^{2s+\a}(B_2\cap \ov\O)$ for $2s+\a>1$,  by Lemma \ref{lem:mapping-prop} and \eqref{eq:change-of-var}  we get:

\begin{lemma} \label{lem:Op-aff}
Let $\cL_\psi$ and $q_A$ be given by \eqref{eq:change-of-var} and  \eqref{eq:q_A}, respectively.  Then,    
   \begin{align*} 
\|\cL_\psi   q_A\|_{C^\a(  \ov {B_{1}^+ })}\leq C   |A| .
 \end{align*}
 with $C$ depending only on $\a,  \|\psi\|_{C^{2s+\a}(B_2)},  N,s$.
 \end{lemma}
 
 \begin{proof}
 It follows from Lemma \ref{lem:mapping-prop} and \eqref{eq:change-of-var}.
 \end{proof}

 Let 
\be \label{eq:def-ell-a}
 \ell=\ell_\a=
 \begin{cases}
  1 \quad \textrm{ for $2s+\a> 1$}\\
   0   \quad\textrm{ for $2s+\a< 1$.}
   \end{cases}
\ee
 For $r>0$ and $w\in L^2_{loc}(\R^N)$,  let  $A_{r,w}\in \R^{N-1}$  such that
\be\label{eq:Aru}
\inf_{A\in \R^N} \|w-\ell A\cdot y'\|_{L^2(B_r^+)}^2= \| w-\ell A_{r,w}\cdot  y'\|_{L^2(B_r^+)}^2.
\ee
Then
\be
\ell A_{r,w} \cdot e_i=\frac{1}{\int_{B_r^+}z_i^2\, dz} \int_{B_r^+} w (y)y_i\, dy.
\ee
In particular
\be\label{eq:perp-Aru}
\int_{B_r^+}( w(y)- \ell A_{r,w} \cdot y' ) y_i\, dy=0 \qquad\textrm{ for all $i=1,\dots,N-1$.} 
\ee
Before stating the main result of the present section, we   record  first the following result from \cite{Fall-regional}.
 
 \begin{lemma}[\cite{Fall-regional}]\label{lemm:appr-esrtim}
 Let $u\in H^s(\R^N_+)$ and $ f\in L^\infty(\R^N_+)$ be a solution to $  \calD_\psi(u,\vp)=\int_{\R^N_+}u\vp\, dx$ for all $\vp\in C^\infty_c(B_2)$. 
  Then, for all $\varrho\in (0,2s)$,  there exists $C=C(N,s,\varrho)>0$ such that 
$$
\|u \|_{C^{2s-\varrho}(B_1)}\leq C \left(  \|u\|_{L ^2(\R^N_+)}+ \|f\|_{L^\infty(\R^N_+)} \right).
$$
\end{lemma}

We next show the following.

\begin{proposition}\label{prop-blow-up}
Let $s\in (0,1)$ and  $\alpha_s\in(0,1-s)$ be given by Lemma~\ref{analysis-beta}.  Let  $\a\in (0, \a_s)$, with $2s+\a\not=1$,     $\psi$ be  as above,  with $\|\psi\|_{C^{\max(1,2s+\a)}(\R^N)}\leq c_0$.  Let $f\in C^{\a}(\R^N_+)$ and $w\in H^s(\R^N_+)\cap C(\R^N_+)$ such that  $w(0)=0$ and 
$$
  \calD_\psi(w,\vp)=\int_{\R^N_+} f \vp\, dx\qquad\textrm{ for all $\vp\in C^\infty_c(B_2)$}.
$$
 Let   $\varrho< 2s+\a-1$ if $2s+\a>1$.
Then,  provided
\be\label{eq:ap-est}
\ell |A_{r,w}|\leq c_0 r ^{-\varrho}, 
\ee
 there exists $C= C(N,s,\a,c_0,\varrho,\b)$ such that 
\be\label{eq:decay-origin}
\sup_{r>0}r^{-2s-\b}\|w - \ell   A_{r,w} \cdot y'\|_{L^\infty(B_r^+)}\leq C (\|w\|_{L^\infty(\R^N_+)}+\|f\|_{ C^\a(\R^N_+)}),
\ee
where $\b=\a-\varrho$,  $\ell=\ell_\a$ and $A_{r,u}$ are given by \eqref{eq:def-ell-a} and \eqref{eq:Aru} respectively.     
\end{proposition}
\begin{proof}
Suppose that \eqref{eq:decay-origin} does not hold.  Then for all integer $n\geq 2$, there exist  $\psi_n\in C^{\max(1,2s+\a)}(\R^N)$, $f_n \in C^{\a}_c(\R^N_+)$ and $w_n \in H^{s}(\R^N_+)\cap C(\R^N_+)$, with $w_n(0)=0$, such that 
\be\label{eq:psi_n-c0}
\|\psi_n\|_{C^{\max(1,2s+\a)}(\R^N)}\leq c_0 ,
\ee
\be 
\|w_n\|_{L^\infty(\R^N_+)}+\|f_n\|_{ C^\a(\R^N_+)}\leq 1
\ee
such that 
\be\label{eq:ap-est-nn}
\ell |A_{r,w_n}|\leq c_0 r ^{-\varrho}, 
\ee
\be \label{eq:D_psi-n}
  \calD_{\psi_n}(w_n,\vp)=\int_{\R^N_+} f_n \vp\, dx\qquad\textrm{ for all $\vp\in C^\infty_c(B_2)$}.
\ee
and
$$
\sup_{r>0}r^{-2s-\a}\| w_n -\ell  A_{r,w_n} \cdot y'\|_{L^\infty(B_r^+)}>n.
$$
Then, since $2s+\b>1$ for $2s+\a>1$,  by a well known (see e.g. \cite{RS16a}),  we can find  sequences $r_n\to 0$ and $\th_n\to \infty$ such that  the function
$$
v_n(y):=\frac{1}{r_n^{2s+\a} \th_n}\left(  w_n(r_ny) -r_n\ell   A_{r_n,w_n} \cdot y' \right)
$$
satisfies 
\be\label{eq:to-contradict}
\|v_n\|_{L^\infty(B_1^+)}\geq \frac{1}{4} 
\ee
and, in  view of \eqref{eq:ap-est-nn},   for all $M\geq 1$    with  $M r_n\leq 1$, 
\be\label{eq:v_n-growth-1}
\|v_n\|_{L^\infty(B_M^+)}\leq  C M^{2s+\b},
\ee
while  for $2s+\a<1$
\be\label{eq:v_n-growth-2}
\|v_n\|_{L^\infty(B_M^+)}\leq  C M^{2s+\b} \qquad\textrm{ for all $M \geq 1$},
\ee
Moreover,  by construction (recalling \eqref{eq:perp-Aru}), 
\be\label{eq:perp}
v_n(0)=0, \qquad\int_{B_{1}^+}v_n(y) y_i\, dy=0 \qquad\textrm{ for all $i=1,\dots,N-1$.} 
\ee
Now for $M r_n>1$ and $2s+\a>1$,  we  have 
$$
\|v_n\|_{L^\infty(B_M^+)}\leq \frac{1}{r_n^{2s+\a} \th_n}(1+ r_n^{-\varrho})\leq M^{2s+\b}+ M^{2s +\b+\varrho}\leq  2 M^{2s+\a}.
$$
In view of this, \eqref{eq:v_n-growth-1} and \eqref{eq:v_n-growth-2}, we see that whenever  $2s+\a\not=1$, 
\be\label{eq:v_n-growth}
\|v_n\|_{L^\infty(B_M^+)}\leq  C M^{2s+\a} \qquad\textrm{ for all $M \geq 1$}.
\ee
Define $\ti\psi_n(x)=\frac{1}{r_n}\psi(r_n x)$ and  let $1<R<\frac{1}{r_n}$.  Then by a change of variable in \eqref{eq:D_psi-n},  
we thus get 
\be\label{eq:final-eq} 
  \calD_{\ti \psi_n}(\chi_Rv_n,\vp)=\int_{\R^N_+} G^n_R \vp\, dx\qquad\textrm{ for all $\vp\in C^\infty_c(B_{\frac{R}{2}})$},
\ee
where 
$$
G^n_R(x):= \frac{1}{r_n^{\b}\th_n} f_n(r_n x)+ c_{N,s} \int_{\R^N_+}\frac{(\chi_R(x)-\chi_R(y)) v_n(y)}{|\ti\psi_n(x)-\ti\psi_n(y)|^{N+2s}} J_n(y)\, dy  - \frac{\ell }{r_n^{\a}\th_n}   \cL_{\psi_n} [   A_{r_n,w_n} \cdot  y'](r_n x),
$$
where $J_n(y):=Jac_{\ti \psi_n}(y) 1_{B^+_{\frac{1}{r_n}}}(y).$
Since $\psi_n$ is a global  diffeomorphism,   by  \eqref{eq:psi_n-c0} and increasing $c_0$ if necessary,  we obtain
\be \label{eq:bi-Lip}
c_0|x-y|\geq  | \ti\psi_n(x)-\ti\psi_n(y) |\geq c_0^{-1}|x-y|, \qquad |D\ti\psi_n(x)| \leq c_0  \qquad\textrm{ for all $x,y\in \R^N$. }
\ee
  Define, for $x\in B_{  R/2}$, 
\begin{align*}
F_R^n(x)&:=c_{N,s}\int_{\R^N_+}\frac{(\chi_R(x)-\chi_R(y)) v_n(y)}{|\ti\psi_n(x)-\ti\psi_n(y)|^{N+2s}} J_n(y)\, dy\\
&= c_{N,s}\int_{\R^N_+\cap \{|y|\geq R\}}\frac{(1-\chi_R(y)) v_n(y)}{|\ti\psi_n(x)-\ti\psi_n(y)|^{N+2s}} J_n(y)\, dy.
\end{align*}
Since $\ti\psi_n(0)=0$,  by the mean value theorem, we have  $F_R^n(x)=F_R^n(0)+\cR_R^n(x)$,  where 
$$
 \cR_R^n(x):=-(N+2s)c_{N,s}\int_0^1 \int_{\R^N_+ }\frac{(\ti\psi_n(x)-\ti\psi_n(y))  \cdot \ti\psi_n(x)}{|\ti\psi_n(x)-\ti\psi_n(y)|} \frac{(1-\chi_R(y)) v_n(y)}{| t \ti\psi_n(x)-\ti\psi_n(y)|^{N+2s+1}} J_n(y)\, dy\, dt.
$$
Note that by \eqref{eq:bi-Lip},   $| t \ti\psi_n(x)-\ti\psi_n(y)|\geq \frac{c_0}{2} |y|$ whenever  $|y|\geq R$ and $|x|\leq \frac{R}{2c_0^2}$.
Therefore, 
by \eqref{eq:v_n-growth} and \eqref{eq:bi-Lip},    there exists $C,n_0>0$ such that
\be\label{eq:cRnR}
 | \n \cR_R^n(x) |\leq {C}  R^{\a-1} \qquad\textrm{ for all $R>1$,  $|x|\leq \frac{R}{2c_0^2}$ and $n\geq n_0$}
\ee 
and  by construction we have  $\cR_R^n(0)=0$.
We can thus write
\be\label{def-GnR}
G^n_R(x)= H^n_R(x)+ c^n_R+\cR^n_R(x), 
\ee
with  
$$
H^n_R(x)= \frac{1}{r_n^{\b}\th_n}\left\{ f_n(r_n x)- f_n(0)+\cL_{\psi_n} [ A_{r_n,w_n} \cdot  y'](r_nx)- \cL_{\psi_n} [  A_{r_n,w_n} \cdot  y'](0)  \right\}  
$$
and 
$$
c^n_R= \frac{1}{r_n^{\b}\th_n}\left\{    f_n(0)+ \cL_{\psi_n} [  A_{r_n,w_n} \cdot  y'](0)  \right\}  + F_R^n(0).
$$
Recalling that $\|f_n\|_{C^\a(\R^N_+)}\leq 1$,   we then get from \eqref{eq:ap-est-nn} and Lemma \ref{lem:Op-aff}  that 
\be\label{def-HnR}
|H^n_R(x)|\leq   \frac{r_n^\a C}{r_n^\b\th_n} + \ell\frac{r_n^{\a-\varrho} }{r_n^ \b\th_n}\leq \frac{C( R)} {\th_n} \qquad\textrm{ for all $x\in B_{R/2}$.} 
\ee
Let us now show  that $c_R^n$ is bounded.  For this, we pick $\phi\in C^\infty_c(  B_{\frac{1}{2c_0^2}}^+)$ with $\int_{\R^N}\phi\, dx=1$.
  Then  multiply \eqref{eq:final-eq} by $\phi$ and integrate over $\R^N_+$ to get
\begin{align}\label{eq:cnR}
|c^n_R|\leq \int_{ \R^N_+}| \chi_R(x) v_n(x)| |\cL_{\ti\psi_n}\phi(x)|\, dx+  \int_{ \R^N_+}| H_R^n(x) \phi(x)|dx+ \int_{ \R^N_+}| \cR^n_R(x) \phi(x)|dx\leq C(R).
\end{align}
In view of this,  \eqref{eq:final-eq},  \eqref{def-GnR}, \eqref{def-HnR} and  \eqref{eq:cRnR}, we can apply Lemma \ref{lemm:appr-esrtim}, to deduce that
$(v_n)_n$ is bounded in $C^{s+\d}_{loc}(\ov {\R^N_+})$,  for some $\d>0$, and converges to some $v\in C^{s+\d}_{loc}(\ov {\R^N_+})\subset H^{s}_{loc}(\ov {\R^N_+})$.  
In addition passing to the limit in \eqref{eq:to-contradict}, \eqref{eq:v_n-growth} and \eqref{eq:perp}, we get 
\be\label{eq:v-growth}
  |v(x)|\leq C (1+|x|^{2s+\a}),   
\ee
\be\label{eq:v-contradiction}
\|v\|_{L^\infty(B_1^+)}\geq \frac{1}{4}, \qquad  v(0)=0, \qquad  
  \int_{B_1^+} v(y) y_i\, dx=0,   \textrm{ for $i=1,\dots,N-1$. }
\ee
Passing to the limit in \eqref{eq:final-eq},  \eqref{eq:cRnR} and using \eqref{eq:cnR},  we get,  for all $\phi\in C^\infty_c(B_{\frac{R c_0^2}{2}})$,
$$
c_{N,s}\int_{\R^{2N}_+}\frac{((\eta_R v)(x)-(\eta_Rv)(y))(\phi(x)-\phi(y))}{|Bx-By|^{N+2s}}\textrm{det}(B)^2 dydx=\int_{\R^N_+}(c_R^\infty+\cR^\infty_R(x))\phi(x)\, dx,
$$
where $B=\lim_{n\to \infty} D\ti\psi_n(0)= \lim_{n\to \infty} D\psi_n(0)$  and 
\be \label{eq:cR-cR}
c^\infty_R\in \R, \qquad   \cR_R^\infty(0) =0, \qquad | \n \cR_R^\infty(x) |\leq {C}  R^{\a-1} \qquad\textrm{ for all $R\geq 2$ and  $x\in B_{\frac{R }{2c_0^2}}$.}
\ee
Recall  from the construction of $\psi_n$ that $B=\textrm{diag}(1,\dots,1,\l)$ is a diagonal matrix, for some $\l\not=0$.  Hence by scaling  and \eqref{eq:cR-cR},  we can assume that 
$$
\Ds_{{\R^N_+}} v \stackrel{\a}{=}0 \qquad\textrm{ in $\R^N_+$}.
$$
Applying  Theorem \ref{lem:from-ND-1D} and using \eqref{eq:v-growth},  we deduce that  $v(x)=b+ a\cdot x'$. 
This is in contradiction with \eqref{eq:v-contradiction}.
\end{proof}

As a consequence, we have :

\begin{corollary}\label{cor-blow-up}
Let $s\in (0,1)$ and  $\alpha_s\in(0,1-s)$ be given by Lemma~\ref{analysis-beta}.  Let  $\a\in (0, \a_s)$,    $\psi$ be  as above,  with $\|\psi\|_{C^{\max(1,2s+\a)}(\R^N)}\leq c_0$.  Let $f\in C^{\a}(\R^N_+)$ and $w\in H^s(B_2^+)\cap L^\infty(\R^N_+)$ such that   
$$
  \calD_\psi(w,\vp)=\int_{\R^N_+} f \vp\, dx\qquad\textrm{ for all $\vp\in C^\infty_c(B_2)$}.
$$
Then,   there exists $C= C(N,s, c_0,\a)$ such that for all $z\in B_{1/2}'$ there exists a vector $e=e_z\in \R^{N-1}=:\de\R^N_+$, with $|e|\leq C$  and 
\be\label{eq:-Taylor-ok}
\sup_{r>0}r^{-2s-\a}\| w-w(z) - e \cdot (y'-z)\|_{L^\infty(B_r^+(z))}\leq C\left( \|w\|_{L^\infty(\R^N_+)}+\|f\|_{ C^{\a}(\R^N_+)} \right).
\ee
In particular,  if $2s+\a>1$ then $\de_{y_N} w=0$ on $B_{1/2}'$.
\end{corollary}
\begin{proof}
Thanks to Lemma \ref{lemm:appr-esrtim}, the function   $v_z(x)=\chi_{2}(x)(w(x+z)-w(z))$  satisfies the hypothesis of Proposition \ref{prop-blow-up} .  We can thus prove \eqref{eq:-Taylor-ok} with $z=0$.  We assume for simplicity that $\|w\|_{L^\infty(\R^N_+)}+\|f\|_{ C^{\a}(\R^N_+)}\leq 1.$\\
If $2s+\a_s<1$, then  the result clearly holds with $e=0$.  We start with \eqref{eq:-Taylor-ok}   and we distinguish the two cases $2s> 1$ and $2s<1$.   We start with the case $2s>1$, where   Lemma \ref{lemm:appr-esrtim},   yields $|w(y)|\leq C|y|$,   so that \eqref{eq:ap-est} holds with $\varrho=0$. 
By \eqref{eq:decay-origin} and an (algebraic) iteration argument (see e.g. \cite{RS16a, Campanato}),  we  can find $e\in \R^{N-1}$ such that
\be\label{eq:2s-eq-1}
 \|w -e\cdot  y'\|_{L^\infty(B_r^+)}\leq C  r^{2s+\a}, \qquad|e|\leq C.
\ee
 We finally consider the case $2s<1$.  Here,    \eqref{eq:decay-origin} implies  that  $|w(y)|\leq C|y|^{2s+\a'}$  as soon as  $2s+\a'<1$ (i.e. $\a'<1-2s<\a$).   As a consequence $w$ satisfies \eqref{eq:ap-est} with $\varrho=1-2s-\a'$.  Therefore  choosing $\a'$ close to $1-2s$ so that $2s+\a-\varrho>1$, we  get from \eqref{eq:decay-origin} and an iteration argument, as above,  that 
 \be\label{eq:2s-eq-1-123}
 \|w -e\cdot  y'\|_{L^\infty(B_r^+)}\leq C r^{2s+\a-\varrho} , \qquad|e|\leq C,
\ee
which again implies that $|w(y)|\leq C |y|$.  Hence  $w$ satisfies \eqref{eq:2s-eq-1-12}   with $\varrho=0$ which, once again by an iteration argument, yields  \eqref{eq:2s-eq-1-123} with $\varrho=0$.

If $2s= 1$  then  Lemma \ref{lemm:appr-esrtim} implies that  $|w(y)|\leq C|y|^{1-\varrho}$,  for all $\varrho\in (0,1)$  so that \eqref{eq:ap-est} holds with   $\varrho \in (0, \a_{\frac{1}{2}})$.   We thus get, from  \eqref{eq:decay-origin} and an iteration argument,   as above,    a vector $e\in \R^{N-1}$ such that
\be\label{eq:2s-eq-1-12}
 \|w -e\cdot  y'\|_{L^\infty(B_r^+)}\leq C r^{1+\a_{\frac{1}{2}}-\varrho} , \qquad|e|\leq C.
\ee
 This gives  $|w(y)|\leq C|y|$, so that  $w$ satisfies \eqref{eq:ap-est}  with $\varrho=0$.    Hence as above,  we get \eqref{eq:2s-eq-1-12} with $\varrho=0$.
 \end{proof}

 We will also need the following interior Schauder estimate.
 
 \begin{lemma}\label{lem:int-Scuader}
Let  $\O$ be an open set containing $B_2$ and $\a\in (0,1)$.  Let $f\in C^{\a}(B_2)$ and  $v\in   H^s(\O)  $ be a solution to
 \begin{equation*}
\Ds_{ \O} v=f\qquad\text{in  $B_{2}$}
	\end{equation*}
satisfying    $C_0:=\int_{\O}\frac{|v(y)|dy}{1+|y|^{N+2s+1}}<\infty$.
Then,  provided $2s+\a\not\in \N$, 
	$$
	\|v\|_{C^{2s+\a}(B_{1})}\leq C(N,s,\a)( C_0+\|v\|_{ L^2(B_2)}+ \|f-f(0)\|_{ C^\a(B_2) }   ).
$$
\end{lemma} 
\begin{proof}
We have 
 \begin{equation*}
\Ds_{ \O}( \chi_{{2}}v)=f+g_v \qquad\text{in  $B_{1}$, }
	\end{equation*}
where for $x\in B_1$
$$
g_v(x)=c_{N,s} \int_{ \O}( 1-\chi_{{2}}(y))|x-y|^{-N-2s}v(y)\,dy.
$$	
Letting $\ti v=  \chi_{{2}}v \in H^s(\R^N)$, we  get 
 \begin{equation*}
\Ds_{ \R^N}\ti  v+\ti vV=f+g_v=c_0+  (f-f(0))+ (g_v-g_v(0))\qquad\text{in  $B_{1/2}$, }
	\end{equation*}
where $c_0=f(0)+ g_v(0)$,   $V(x)=c_{N,s}\int_{\R^N\setminus\O}|x-y|^{-N-2s}\,dy$.
Clearly  $\|V\|_{C^3( B_{1})}\leq C(N,s)$.  Moreover   $\|g_v-g_v(0)\|_{C^1( B_{1/2})}\leq C(N,s) C_0$.   Now as in Remark \ref{rem-gen-to-ok}$(ii)$, we can estimate
\begin{align*}
|c_0|&\leq C(N,s)\left( C_0+ \|\ti v V\|_{L^1( B_{1/4})} +  \|g_v-g_v(0)\|_{L^1( B_{1/4})}+\|f-f(0)\|_{L^1( B_{1/4})}  \right)\\
&\leq C(N,s) \left( C_0+ \|v \|_{L^2( B_{2})} +  \|f-f(0)\|_{L^\infty( B_{2})} \right).
\end{align*}
By the interior regularity estimates in \cite{Fall-reg-1}, we have that 
	$$
	\|\ti v\|_{C^{2s-\e}(B_{1/4})}\leq C(N,s,\e)( C_0+\|v\|_{ L^2(B_2)}+ \|f-f(0)\|_{ L^\infty(B_2) }   ).
$$
We can now apply the interior  Schauder  estimate in \cite{RS16a} and a standard bootstrap (only necessary for $2s<1$ because $\a<1$) argument to get the desired estimate.
\end{proof}

We can now prove the following.

\begin{theorem}\label{th-1}
Let $s\in(0,1)$ and  $\alpha_s\in(0,1-s)$ be given by Lemma~\ref{analysis-beta}, which satisfies $2s+\alpha_s>1$.
Let $\Omega\subset \R^N$ be a $C^{\max(1,2s+\alpha)}$ domain, with $0\in \partial\Omega$, and let $f\in C^\alpha(\overline\Omega \cap B_1)$, with $\alpha<\alpha_s$ and with $\alpha+2s\neq1$.
Let $u\in H^s(\O)$ be such that 
$$
 D_\O(u,\vp)=\int_{\O}f\vp\, dx \qquad \textrm{$\forall\vp \in C^\infty_c(B_1)$.}
$$
Then,
\be \label{eq:schaud-Neum}
\|u\|_{C^{2s+\a}(\ov \O\cap B_{1/2})}\leq C\left(\|u\|_{L ^2(\O)}+ \|f\|_{C^\a(\ov \O \cap B_1)}  \right).
\ee
with $C$ depending only on $N$, $s$, $\alpha$, and $\Omega$.

Moreover, if $2s+\a>1$ then $\de_\nu u=0$ on $\de\O\cap B_1$, where $\nu$ is the exterior normal of $\de\O$. 
\end{theorem}

\begin{proof}
%
We assume for simplicity that $\|u\|_{L^2(\O)}+\|f\|_{C^{\a}(\ov \O)}\leq 1$ and that $\O$ satisfies \eqref{eq: diffeom} and \eqref{eq:de-n-psi-normal}.

By Corollary \ref{cor-blow-up} and the  change of variable $w(y)=u(\psi(y))$  we have that, for all $\z\in \de\O \cap B_{1/2}$, there exists $e=e_\z\in \R^{N-1}=:\de\R^N_+$ with $|e_\z|\leq C$ and  such that 
\be \label{taylor-exp-bdr}
\|u-u(\z)- \ell_{\a} e\cdot (\psi^{-1}(\cdot)-\z)\|_{L ^\infty(B_r(\z)\cap\O)}\leq C r^{2s+\a}.
\ee
In addition,  using also  \eqref{eq:de-n-psi-normal},  we have  $\de_\nu u=0 $ on $\de\O\cap B_{1/2}$ as soon as $2s+\a>1$.\\

Let $x_0\in \O\cap B_1$ and $ \z\in \de \O$ be such that $\textrm{dist}(x_0,\de\O)=|x_0-\z|$. Put $\rho:=\frac{|x_0-\z|}{4}$ and $\ov x=\frac{x_0-\z}{2}\in \O$ then  $B_{2\rho}(\ov x)\subset\O$.  We next, define 
$$
v_\rho(y)= (u(\ov x+\rho y)-u(\z))- \ell_{\a} e\cdot (\psi^{-1}(\rho y+\ov x)-\z),  \qquad f_\rho(y)=\rho^{2s}f(\ov x+\rho y)
$$
and $\O_\rho=\frac{1}{\rho}(\O-\ov x )$.
Then 
\be\label{eq:v-rho-int}
\Ds_{\O_\rho} v_\rho=   f_\rho+    F_\rho=:G_\rho  \qquad\textrm{ in $B_{2}\subset \O_\rho$},
\ee
where $F_\rho (y):=\ell_{\a} \rho^{2s} [ \Ds_{\O}(e\cdot \psi^{-1})](\ov x +\rho y))$.
Since $f\in C^{\a}(\O)$,  using  Lemma \ref{lem:mapping-prop},  we  then  get $[G_\rho]_{C^{\a}(B_{1})}\leq C \rho^{2s+\a}$.    By \eqref{taylor-exp-bdr}, we have  that $\|v_\rho\|_{L^{\infty}(B_{1})}\leq C \rho^{2s+\a}$ and 
\be \label{eq:v-rho-near-fin}
\|v_\rho (\cdot/\rho) \|_{L^\infty(B_{R\rho }\cap (\O-\ov x))}\leq \|v_\rho (\cdot/\rho) \|_{L^\infty(B_{R\rho+2\rho }\cap(\O-\z))}\leq C (R \rho)^{2s+\a} \qquad\textrm{ for all $R\geq 1$. }
\ee
 Now applying     Lemma \ref{lem:int-Scuader} to \eqref{eq:v-rho-int},  we obtain
$$
\| v_\rho\|_{C^{2s+\a}(B_{1/4})}\leq C(N,s,\a)\left(  \int_{\O_\rho\setminus B_{1}}\frac{|v_\rho(y)|dy}{|y|^{N+2s+1}}+\|v_\rho \|_{L^{\infty}(B_{1})} + \| G_\rho -G_\rho(0)\|_{C^{\a}(B_{1})} \right).
$$
Thanks to \eqref{eq:v-rho-near-fin} and a change of variable we get 
\begin{align*}
& \int_{\O_\rho\setminus B_{1}}\frac{|v_\rho(y)|dy}{|y|^{N+2s+1}} =\rho^{2s+1} \int_{(\O-\ov x)\setminus B_{\rho}}\frac{|v_\rho(x/\rho )|dx}{|x|^{N+2s+1}}\\
& \leq \sum_{i=0}^\infty  (2^i\rho)^{-N-2s-1} \rho^{2s+1}\int_{(\O-\ov x )\{|y|<2^{i+1}\rho \}}{|v_\rho(x/\rho)|,dx} \\
 & \leq C\sum_{i=0}^\infty  (2^i\rho)^{-N-2s-1} \rho^{2s+1} (2^{i+1}\rho)^{N+2s+\a}\leq  C \rho^{2s+\a}.
\end{align*}
We conclude that $\| v_\rho\|_{C^{2s+\a}(B_{1/4})}\leq C  \rho^{2s+\a}.$
Therefore scaling  and translating back, we get for $2s+\a< 1$ that
\be
[u]_{C^{2s+\a}(B_{\rho}(\ov x))} \leq C
\ee
and thus  $[ u]_{C^{2s+\a}(B_{1/2}\cap\O)} \leq C$ because $\ov x$ is arbitrary.\\

If now $2s+\a>1$, we obtain
$$
[\n u- \n {e}\cdot(\psi^{-1}(\cdot)-\z)]_{C^{2s+\a-1}(B_{\rho}(\ov x))} \leq C.
$$
Hence, since  $[\n {e} \cdot(\psi^{-1}(\cdot)-\z)]_{C^{2s+\a-1}(B_{\rho}(\ov x))} \leq C$, we get 
$$ 
[\n u]_{C^{2s+\a-1}(B_{\rho}(\ov x))} \leq C.
$$
This implies 
$$ 
[\n u]_{C^{2s+\a-1}(B_{1/2}\cap\ov \O)} \leq C.
$$
\end{proof}

Finally, we give the:

\begin{proof}[Proof of Theorem \ref{thm-Neu}]
The result follows from Theorem \ref{th-1}.
\end{proof}

\section{Regional Dirichlet problem }
\label{sec6}

In this section, we consider $\phi \in C^{2, \b}(\R^{N-1})$ satisfying  $\|\phi\|_{ C^{2,\b}(\R^{N-1})}\leq \frac{1}{4}$ and $\phi(0)=|\n\phi(0)|=0$. 
We define $\Phi(x',x_N)=(x',x_N+\phi(x'))$.
We start with the following result.
\begin{lemma}\label{lem:comput-L-prof}
Let $\eta\in C^{2s-1}_c(\R)$ with $\eta=0$ on $\R_-$ and $\ell(x'):=a\cdot x'+b$.  Let $f(x)=\eta(x_N)\ell(x')$. Then,  for all $x\in \R^N_+$, 
\begin{align*}
&c_{N,s}\int_{\R^N_+}\frac{f(x)-f(y)}{|\Phi(x)-\Phi(y)|^{N+2s}}\, dy=\ell(x')T_1(x')\Ds_{\R_+}\eta(x_N)+ \ell(x')T_3(x')   \Ds_{\R_+}(x_N \eta) \\
&-   T_2(x') \cdot a \int_{\R}\frac{(t-x_N)(\eta(t)-\eta(x_N))}{|x_N-t|^{1+2s}}\, dt  -T_3(x') \ell(x')\left(  \eta \Ds_{\R_+}x_N+ x_N \Ds_{\R_+}\eta  \right)+ h(x),
\end{align*}
where $\|h\|_{ C^{\min(\a,\b)} (\ov{B_1^ +})}\leq C\|\eta\|_{C^{2s-1}(\R)} $ for all $\a<2s-1$,  $T_1(0)=1$, $T_2(0)=0$ and 
$\|T_j\|_{ C^{1,\b} (B_1')}\leq C(N,s,\b)  $, for $j=1,2$. Moreover,   $\|T_3\|_{ C^{\b} (B_1')}\leq C(N,s,\b)  $ and 
\be\label{eq:C30-MC}
T_3(0)= -  D^2\phi(0)[e_i,e_i]= -  \D\phi(0).
\ee
%
\end{lemma}
\begin{proof}
We have 
\begin{align*}
&\int_{\R^N_+}\frac{f(x)-f(y)}{|\Phi(x)-\Phi(y)|^{N+2s}}\, dy =\int_{\R^{N-1}}\int_{-x_N}^\infty\frac{f(x)-f(x+z)}{|\Phi(x)-\Phi(x+z)|^{N+2s}}\, dz\\
&=\int_{\R^{N-1}}\int_{-x_N}^\infty\frac{f(x)-f(x+z)} {|D\Phi(x')z |^{N+2s}}\, dz+ \int_{\R^{N-1}}\int_{-x_N}^\infty\frac{f(x)-f(x+z))}{|z|^{N+2s}} B(x',|z'|,z'/|z'|, z/|z|)\, dz,
\end{align*}
where 
$$
 B(x',r,\th',\th)=\frac{1}{| \int_{0}^1D\Phi(x'+t r\th' )\th dt|^{N+2s}}-  \frac{1}{|D\Phi(x')\th |^{N+2s}}.
$$
Recall that  for all $y\in \R^N$
$$
 D\Phi(x+ r\th' )y=y+\left( \n\phi(x'+r\th')\cdot y'  \right)e_N.
$$
We can thus write
\be \label{eq:Bsrt}
  B(x',r,\th',\th)= r \mu_1(x',\th',\th)+r  \calO(x',r,\th',\th),
\ee
where 
\begin{align*}
\mu_1(x',\th',\th)=  \de_r B(x',0,\th',\th)&=-(N+2s)\frac{D^2\Phi(x')[\th',\th-(\th\cdot e_N) e_N]\cdot D\Phi(x')\th}{|D\Phi(x')\th|^{N+2s+2}},
%
\end{align*}
and 
$$
 \calO(x',r,\th',\th)=\de_r B(x',r,\th',\th)-\de_r B(x',0,\th',\th).
$$
We then have 
$\mu_1, \calO \in C^\b_x(B_2)$, $\mu_1(x',-\th',-\th)=-\mu_1(x',\th',\th)$,  $|\calO(x',r,\th',\th)|\leq C\min (1, r^\b)$ and 
\be \label{eq:calO}
 |\calO(x',r,\th',\th)|\leq  C\min (1, r^\b), \qquad |\calO(x',r,\th',\th)-\calO(\ov x',r,\th',\th)|\leq  C\min (|x'-\ov x'|, r^\b).
\ee
We thus get from \eqref{eq:Bsrt}, 
\begin{align}
&c_{N,s}\int_{\R^N_+}\frac{f(x)-f(y)}{|\Phi(x)-\Phi(y)|^{N+2s}}\, dy  =c_{N,s}\int_{\R^{N-1}}\int_{-x_N}^\infty\frac{f(x)-f(x+z)} {|D\Phi(x')z |^{N+2s}}\, dz \nonumber\\
&+ c_{N,s}\int_{\R^{N-1}}\int_{-x_N}^\infty\frac{f(x)-f(x+z))}{|z|^{N+2s}}  |z'| \mu_1(x',z'/|z'|,z/|z|)\, dz \nonumber\\
&+c_{N,s}\int_{\R^{N-1}}\int_{-x_N}^\infty \frac{f(x)-f(x+z))}{|z|^{N+2s}}  |z'|\calO(x',|z'|,z'/|z'|,z/|z|)\, dz \nonumber\\
%
%
&=I_1(x)+I_2(x)+ I_3(x).
\label{eq:I1I2I3}
\end{align}
Letting $T_1(x'):=\frac{1}{c_{1,s}}\int_{\R^{N-1}}\frac{c_{N,s}}{|D\Phi(x')(z',1) |^{N+2s}}\, dz'$  and $T_2(x')\cdot e_i=c_{N,s}\int_{\R^{N-1}}\frac{z_i}{|D\Phi(x')(z',1) |^{N+2s}}\, dz'$,
we then get 
\begin{align*}
I_1(x)&=\ell(x')\int_{\R^{N-1}}\int_{-x_N}^\infty\frac{\eta(x_N)-\eta(x_N+z_N)} {|D\Phi(x')z |^{N+2s}}\, dz -\int_{\R^{N-1}}\int_{-x_N}^\infty\frac{a\cdot z' \eta(x_N+z_N)} {|D\Phi(x')z |^{N+2s}}\, dz\\
&=T_1(x') \ell(x')c_{1,s}\int_{-x_N}^\infty \frac{\eta(x_N)-\eta(x_N+z_N)}{|z_N|^{1+2s}} \, dz_N-T_2(x')\cdot a \int_{-x_N}^\infty z_N  \frac{\eta(x_N+z_N)}{|z_N|^{1+2s}}\, dz_N\\
&={T_1(x')}{} \ell(x')\Ds_{\R_+}\eta(x_N)- T_2(x')\cdot a\int_{\R^+}\frac{(t-x_N)\eta(t)}{|x_N-t|^{1+2s}}\, dt.
\end{align*}
Since $\eta=0$ on $\R_-$,   then
\begin{align*}
\int_{\R^+}\frac{(t-x_N)\eta(t)}{|x_N-t|^{1+2s}}\, dt&=\int_{\R}\frac{(t-x_N)\eta(t)}{|x_N-t|^{1+2s}}\, dt =\int_{\R}\frac{(t-x_N)(\eta(t)-\eta(x_N))}{|x_N-t|^{1+2s}}\, dt.
%
%
\end{align*}
We thus conclude that 
\begin{align} \label{eq:I1}
I_1(x)& ={T_1(x')}  \ell(x')\Ds_{\R_+}\eta(x_N)- T_2(x')\cdot a\int_{\R}\frac{(t-x_N)(\eta(t)-\eta(x_N))}{|x_N-t|^{1+2s}}\, dt.
\end{align}
Since $D \Phi(0)=id$,   we see that $  {T_1(0)}=\frac{1}{c_{1,s}}\int_{\R^{N-1}}\frac{c_{N,s}}{(1+|z'|^2)^{\frac{N+2s}{2}}} dz'=1$ and by  oddness we have $T_2(0)=0$.  \\

 We consider next  $I_2$.
Similarly, letting $T_3(x')=\frac{c_{N,s}}{ c_{1,s}}\int_{\R^{N-1}}\frac{|z'|}{|(z',1)|^{N+2s}} \mu_1(x', \frac{z'}{|z'|}, \frac{(z',1)}{|(z',1)|})\, dz'$, we obtain
\begin{align}\label{eq:I2-1}
I_2(x)=& T_3(x')\ell(x') c_{1,s}\int_{-x_N}^\infty z_N\frac{\eta(x_N)-\eta(x_N+z_N)}{|z_N|^{1+2s}} \, dz_N+ h_1(x) \nonumber\\
& = T_3(x') \ell(x')c_{1,s}\int_{\R^+}\frac{(t-x_N)(\eta(x_N)-\eta(t))}{|x_N-t|^{1+2s}}\, dt+h_1(x) \nonumber\\
&={T_3(x')}  \ell(x')\left( \Ds_{\R_+}(x_N\eta)-\eta \Ds_{\R_+}x_N- x_N \Ds_{\R_+}\eta  \right) +h_1(x),
\end{align}
where 
\begin{align*}
h_1(x)&:=c_{N,s} \int_{\R^{N-1}}\int_{-x_N}^\infty \frac{|z'|a\cdot z' \eta(x_N+z_N)}{|z|^{N+2s}}  \mu_1(x',z'/|z'|,z/|z|)\, dz\\
&=c_{N,s}\int_{\R^{N-1}}\frac{|z'|a\cdot z'}{|(z',1)|^{N+2s}} \mu_1(x', \frac{z'}{|z'|},\frac{(z',1)}{|(z',1)|})\, dz'\int_{-x_N}^\infty \frac{ z_N\eta(x_N+z_N)}{|z_N|^{2s}}\, dz_N. 
\end{align*}
Since $2s>1$,  we can easily see that
\be  \label{eq:h_1}
\|h_1\|_{C^{\min(\b,2s-1)}(\ov{B_1^+})}\leq C |a|\|\eta\|_{C^{2s-1}(\R)}.
\ee
Using that $D\Phi(0)=Id$, a change of variable and integration by part, we get
\begin{align}
&c_{1,s}T_3(0)=-(N+2s)   {c_{N,s}}{ } \int_{\R^{N-1}}\frac{|z'|^2D^2\phi(0)[z'/|z'|,z'/|z'|]}{|(z',1)|^{N+2s+2}}  \, dz' \nonumber\\
&=-(N+2s) {c_{N,s}}{ } \int_0^\infty\frac{r^{N}}{(r^2+1)^{\frac{N+2s+2}{2}}}\, dr \sum_{i=1}^{N-1} D^2\phi(0)[e_i,e_i]\int_{S^{N-2}}\th_1^2d\s(\th) \nonumber\\
&=-(N+2s) \frac{c_{N,s} |S^{N-2}|}{N-1 } \int_0^\infty\frac{r^{N}}{(r^2+1)^{\frac{N+2s+2}{2}}}\, dr  \sum_{i=1}^{N-1} D^2\phi(0)[e_i,e_i] \nonumber\\
&=- {c_{N,s} |S^{N-2}|} \int_0^\infty\frac{r^{N-2}}{(r^2+1)^{\frac{N+2s}{2}}}\, dr  \sum_{i=1}^{N-1} D^2\phi(0)[e_i,e_i] =- c_{1,s}  \sum_{i=1}^{N-1} D^2\phi(0)[e_i,e_i].
\label{eq:C30}
\end{align}
Here, we used that ${c_{N,s}}{}\int_{\R^{N-1}}\frac{1}{(1+|z'|^2)^{\frac{N+2s}{2}}} dz'=c_{1,s}$.
Finally,  using \eqref{eq:calO}, we see that  
$$
\|I_3\|_{C^{\min(\a,\b)}(\ov{B_1^+})}\leq C\|f\|_{C^{2s-1}(\R)} \qquad\textrm{ for all $\a\in (0,2s-1)$.}
$$

From this together with  \eqref{eq:C30}, \eqref{eq:h_1}, \eqref{eq:I2-1},  \eqref{eq:I1} and \eqref{eq:I1I2I3}, we get the result.
\end{proof}

\begin{lemma}\label{lem:Dir-Op-bdr}
Let $\ov \eta\in C^{\infty}_c(-3,3)$ with $\ov \eta=1$ on $(-2,2)$ and $\ell(x'):=a\cdot x'+b$.  Define
$$
P(x)=\ov \eta(x_N) (x_N)^{2s-1}_+\ell(x')- b T_3(0) \ov \eta(x_N) (x_N)^{2s}_+  ,
$$
where $T_3(0)$ is given by Lemma \ref{lem:comput-L-prof}.
 Then, for all $x\in B_1^+$, 
\begin{align*}
&c_{N,s}\int_{\R^N_+}\frac{P(x)-P(y)}{|\Phi(x)-\Phi(y)|^{N+2s}}\, dy= g_1(x)\log(x_N)+ g_2(x),
\end{align*}
where $\|g_i\|_{ C^{\min(\b,\a)} (\ov{B_1^+})}\leq C(N,s,\b, \a,\ov\eta) (|a|+|b|)$ for all $\a<2s-1$ and   $g_1(0)=0$.
\end{lemma}

\begin{proof}
Recall from Lemma \ref{lem:comput-L-prof} that 
\begin{align}\label{eq:L1L2}
&c_{N,s}\int_{\R^N_+}\frac{P(x)-P(y)}{|\Phi(x)-\Phi(y)|^{N+2s}}\, dy= L_1(x)+L_2(x)+h(x),
\end{align}
where  $h\in C^\b(\ov{B_1^+})$,
\begin{align*}
L_1(x)&= \ell(x')  T_1(x')\Ds_{\R_+}(x_N^{2s-1}\ov \eta)+ \ell(x') T_3(x')   \Ds_{\R_+}(x_N^{2s}\ov \eta)\\
&-b T_3(0) T_1(x')    \Ds_{\R_+}(x_N^{2s}\ov \eta)-b  T_3(0)T_3(x')   \Ds_{\R_+}(x_N^{2s+1}\ov \eta)
\end{align*}
and 
\begin{align}
L_2 (x) &=  -   T_2(x') \cdot a \int_{\R}\frac{(t-x_N)(t^{2s-1}_+\ov \eta(t)-(x_N)^{2s-1}_+\ov \eta(x_N))}{|x_N-t|^{1+2s}}\, dt \nonumber\\
& -T_3(x') \ell(x')\left(  \ov\eta x_N^{2s-1} \Ds_{\R_+}x_N+ x_N \Ds_{\R_+}(x_N^{2s-1}\ov \eta)  \right) \nonumber\\
&-   b T_3(0)T_2(x') \cdot a \int_{\R}\frac{(t-x_N)(t^{2s}_+\ov \eta(t)-(x_N)^{2s}_+\ov \eta(x_N))}{|x_N-t|^{1+2s}}\, dt \nonumber\\
&-bT_3(0)T_3(x') \left(  \ov\eta x_N^{2s} \Ds_{\R_+}x_N+ x_N \Ds_{\R_+}(x_N^{2s}\ov \eta)  \right). \label{eq:L2}
\end{align}
 We now write 
\begin{align*}
&L_1(x)=- bT_3(0)( T_1(x')-1) \Ds_{\R_+}(x_N^{2s}\ov \eta)+ b(T_3(x') -T_3(0)) \Ds_{\R_+}(x_N^{2s}\ov \eta)\\
&+a\cdot x' T_3(x')   \Ds_{\R_+}(x_N^{2s}\ov \eta)-b T_3(x')  \Ds_{\R_+}(x^{2s+1}_N \ov\eta)  +  \ell(x')  T_1(x')\Ds_{\R_+}(x_N^{2s-1}\ov \eta).
\end{align*}
%
%
Moreover since $x_N\mapsto  x^{2s+1}_N \ov\eta(x_N) $ has zero derivative at $x_N=0$, we get  $\Ds_{\R_+}(x^{2s+1}_N \ov\eta) \in C^\b([0,1])$ by Lemma \ref{lem:mapping-prop}.   In addition,  since  $ \Ds_{\R_+} x^{2s-1}=0$, we see that $\Ds_{\R_+} (x^{2s-1})\ov\eta\in C^\b([0,1])$. Therefore,  recalling that $T_1(0)=1$, 
we thus obtain
\be \label{eq:L1}
L_1(x)=g_1(x)H(x_N)+ g_2(x),
\ee
for some functions $g_i$ as in the statement of the lemma and $H(x_N)=\Ds_{\R_+}(x_N^{2s}\ov \eta)(x_N)$.
To estimate $L_2$,  we first observe that 
\begin{align*}
 &c_{1,s}\int_{\R}\frac{(t-x_N)(t^{2s-1}_+\ov \eta(t)-(x_N)^{2s-1}_+\ov \eta(x_N))}{|x_N-t|^{1+2s}}\, dt \\
 &=-\Ds (x_N( x_N)^{2s-1}_+\ov \eta)+x_N \Ds ( (x_N)^{2s-1}_+\ov \eta)\\
 &=-\Ds_{\R_+} ( (x_N)^{2s}\ov \eta)- a_s +x_N  c_{N,s}\int_{\R_+}\frac{ \ov \eta(x_N)- \ov \eta(t))t^{2s-1}}{|x_N-t|^{1+2s}}\, dt,
\end{align*}
where we used that $ \Ds_{\R_+} x_N^{2s-1}=0$ and $\Ds x_N=0$ and the formula $\Ds(1_{\R_+}u)(x)=\Ds_{\R_+}u+a_s x^{-2s} u$ in $\R_+$.  
We have that 
\begin{align}\label{eq:grad-grad}
  c_{1,s}&\int_{\R}\frac{(t-x_N)(t^{2s}_+\ov \eta(t)-(x_N)^{2s}_+\ov \eta(x_N))}{|x_N-t|^{1+2s}}\, dt=-\Ds( (x_N)_+^{2s+1}\ov \eta)+x_N\Ds (x_N)^{2s}_+ \ov\eta \nonumber\\
  &=-\Ds((x_N)_+^{2s+1}\ov \eta)+x_N\Ds_{\R_+} (x_N^{2s} \ov\eta)+ a_s x_N.
\end{align}
Moreover $\Ds_{\R_+} x_N=\k_s x_N^{1-2s}$, for some $\k_s>0$.  Next,  using that  $ \Ds_{\R_+} x^{2s-1}=0$, we see that $\Ds_{\R_+} (x^{2s-1}\ov\eta)\in C^\b([0,1])$.  Since $T_2(0)=0$,  we thus get from \eqref{eq:grad-grad}  and \eqref{eq:L2}, 
\be   \label{eq:L2}
L_2(x)=g_1(x)H(x_N)+ g_2(x),
\ee
for some functions $g_i$ as in the statement of the lemma and $H(x_N)=\Ds_{\R_+}(x_N^{2s}\ov \eta)(x_N)$.\\
To conclude, we observe that 
\begin{align*}
&H'(x_N)=2s \Ds (x_N^{2s-1}\ov \eta)(x_N)+  \Ds (x_N^{2s}\ov \eta')(x_N)-a_s\ov\eta'(x_N)\\
&=2s \Ds_{\R_+} (x_N^{2s-1}\ov \eta)(x_N)+ a_s x_N^{-1}\ov\eta(x_N)+  \Ds (x_N^{2s}\ov \eta')(x_N)-a_s\ov\eta'(x_N).
\end{align*}
As a consequence, $x_N\mapsto H'(x_N)-a_s x_N^{-1}\in C^\infty([0,1])$.  It follows that for $x_N\in (0,1)$, we have  $H(x_N)=-\int_{x_N}^{1/2}H'(t)\, dt+H(1/2)=a_s\log x_N+\ti h(x_N)$, with  $\ti h\in C^\infty([0,1])$. In view of this,  \eqref{eq:L2}, \eqref{eq:L1} and \eqref{eq:L1L2}, we get the result.
\end{proof}

\subsection{Expansion near the boundary}
%
In the remaining of this section, we define  $\calB_\Phi:H^s(\R^N_+)\times H^s(\R^N_+)\to \R$ by 
$$
  \calB_\Phi(u,\vp):=  \frac{c_{N,s}}{2} \int_{\R^N_+\times \R^N_+}\frac{(u(x)-u(y))(\vp(x)-\vp(y))}{|\Phi(x)-\Phi(y)|^{N+2s}}\, dxdy.
$$
We start recalling the following result from \cite{Fall-regional}.

\begin{lemma}[\cite{Fall-regional}]\label{lemm:appr-esrtim-Dir}
Let $2s>1$.
Let $f\in L^p(B_2^+)$  and $u\in H^s_0(\R^N_+)$ by a solution to 
$$
  \calB_\Phi(u,\vp)= \int_{\R^N_+} f \vp\, dx\qquad\textrm{ for all $\vp\in C^\infty_c(B_2^+)$}.
$$
  Then for all $p>N$, there exists $C=C(N,s,p )>0$ such that 
$$
\|u/x_N^{2s-1}\|_{C^{1-N/p}(B_1^+  )}\leq C \left(  \|u\|_{L ^2(\R^N_+)}+ \|f\|_{L^p(B_2^+)} \right).
$$
\end{lemma}
We state the next result which is crucial for the proof of Theorem \ref{th-2}.
\begin{proposition}\label{prop-blow-u-Dir}
Let  $s\in(\frac12,1)$,  $\alpha_s\in(0,1-s)$ be given by Lemma~\ref{analysis-beta} and $\b\in (0,\a_s)$. Let   $\a\in (0,\min(\b, 2s-1))$, with $2s+\a\not=1$,   $f\in C^\a(\R^N_+)$ and $w\in H^s_0(\R^N_+) $ such that   
$$
  \calB_\Phi(w,\vp)=\int_{\R^N_+} f \vp\, dx\qquad\textrm{ for all $\vp\in C^\infty_c(B_2^+)$}
$$
and 
$$
\|w/x_N^{2s-1}\|_{L^\infty(\R^N_+)}+\|f\|_{ C^\a(\R^N_+)} \leq 1.
$$
Let $\ov \eta\in C^{\infty}_c(-3,3)$ with $\ov \eta=1$ on $(-2,2)$.
Then,    there exists $C= C(N,s,\a,\b,\ov \eta)$ such that 
\be\label{eq:to-contra-bdr}
\sup_{r>0}r^{-1-\a }\left\|\frac{w}{x_N^{2s-1}} -\left( a_{r,w} \cdot x' +b_{r, {w}} (1-T_3(0)x_N) \right)    \ov\eta (x_N) \right\|_{L^\infty(B_r^+)}\leq C ,
\ee
where  $a_{r,w} :=A_{r, \,w/{x_N^{2s-1}} }$ is given by \eqref{eq:Aru},  
$$
b_{r, {w}}=\frac{1}{\int_{B_{r }^+}(1- T_3(0)y_N)^2dy  }\int_{B_{r}^+} \frac{w}{x_N^{2s-1}}(y) (1- T_3(0)y_N)dy,
$$
 and  $T_3(0)=-\D\phi(0) $ is given by \eqref{eq:C30-MC}.
\end{proposition}
\begin{proof}
Suppose on the contrary  that \eqref{eq:to-contra-bdr} does not hold.  Then for all integer $n\geq 2$, there exist  $\phi_n\in C^{2,\b}(\R^{N-1})$, $f_n \in C^\a(\R^N_+)$ and $w_n \in H^{s}_0(\R^N_+) $,  such that 
\be\label{eq:psi_n-c0}
\|\phi_n\|_{C^{2,\b}(\R^{N-1})}\leq \frac{1}{4} ,\qquad\phi_n(0)=|\n \phi_n(0)|=0,
\ee
\be \label{eq:w_nf_n}
\|w_n/x_N^{2s-1}\|_{L^\infty(\R^N_+)}+\|f_n\|_{ C^\a(\R^N_+)}\leq 1
\ee
with
\be\label{eq:w_nvp} 
  \calB_\Phi(w_n,\vp)=\int_{\R^N_+} f_n \vp\, dx\qquad\textrm{ for all $\vp\in C^\infty_c(B_2^+)$}.
\ee
and
\be \label{eq:absurd}
 \sup_{r>0}r^{-1-\a }\left\|\frac{w_n}{x_N^{2s-1}} - \left( a_{r,w_n} \cdot x' +b_{r, {w_n}} (1-T_3^n(0)x_N) \right)  \ov\eta(x_N) \right\|_{L^\infty(B_r^+)}>n,
\ee
where (recalling \eqref{eq:C30-MC})
\be\label{eq:C30-MC-n}
T_3^n(0):= - \D\phi_n(0) .
\ee
Define 
$$
Q_r(x)= a_{r,w_n} \cdot x' +b_{r, {w_n}} (1-T_3^n(0)x_N) .
$$
By \eqref{eq:w_nf_n},  we can define the monotone nonincrease sequence of function $\th_n$ by
\be\label{eq:def-th-n}
\th_n(\rho):=\sup_{r>\rho}r^{-1-\a}\left\|\frac{w_n}{x_N^{2s-1}} - Q_r(x)  \ov\eta(x_N) \right\|_{L^\infty(B_r^+)}. 
\ee
Obviously, for $n\geq 2$,  by \eqref{eq:absurd},   there exists $\rho_n>0$ such that 
\be\label{eq:The-n-geq-n}
  \th_n(\rho_n)> n/2\geq 1 .
\ee
Hence,  provided $n\geq 2$,  by definition of $\th_n(\rho_n)$ as a supremum and the monotonicity of $\th_n$,  there exists $r_n\geq \rho_n $ such that 
\begin{align}
\th_n( r_n)&\geq   r^{-1-\a}_n\left\|\frac{w_n}{x_N^{2s-1}} - Q_{r_n}(x)  \ov\eta(x_N) \right\|_{L^\infty(B_{r_n}^+)} \nonumber\\
& \geq \th_n( \rho_n)-1/2\geq (1-1/2)\th_n(\rho_n)\geq \frac{1}{2}\th_n( r_n),
\label{eq:th_n-use}
\end{align}
Also by \eqref{eq:The-n-geq-n} and the above estimate,  we have   that $\th_n(r_n)\geq \frac{n}{2}-\frac{1}{2}\to \infty$ as $n\to \infty.$ 
We now define $v_n\in H^s_0(\R^N_+)$ by
$$
v_n(x):=\frac{1}{r_n^{2s+\a } \th_n(r_n)}\big(  w_n(r_nx) -Q_{r_n}(r_n x)\ov\eta(r_n x_N)  (r_nx_N)^{2s-1}  \big).
$$
In view of \eqref{eq:th_n-use},  it clearly satisfies 
\be\label{eq:to-contradict-Dir}
\|v_n/x_N^{2s-1}\|_{L^\infty(B_1^+)}\geq \frac{1}{2} .
\ee
Moreover,  by construction (recalling \eqref{eq:perp-Aru}), 
\be\label{eq:perp-Dir}
\int_{B_{1}^+} \frac{v_n}{x_N^{2s-1}}(y)  (1-T_3^n(0)y_N)dy=0, \qquad\int_{B_{1}^+}\frac{v_n}{x_N^{2s-1}}(y) y_i\, dy=0 \qquad\textrm{ for all $i=1,\dots,N-1$.} 
\ee
Note that by Lemma \ref{lemm:appr-esrtim-Dir},  \eqref{eq:w_nvp}  and \eqref{eq:w_nf_n},
\be\label{eq:ap-est}
 |a_{r,w_n}|\leq c_0 r ^{-\varrho},  \qquad\textrm{for all $\varrho\in (0,\b-\a)$.}
\ee
By \eqref{eq:psi_n-c0} and \eqref{eq:C30-MC-n},  for $M\leq 1$,    we have 
\begin{align}\label{eq:Q2rQr}
\| Q_{2r}-Q_r\|_{L ^2(B_M)}^2\asymp  |B_M| \left(M^2 |a_{2r,w_n}-a_{r,w_n}|^2+  |b_{2r,w_n}-b_{r,w_n}|^2  \right).
\end{align}
Therefore  by \eqref{eq:def-th-n},  for $2r\leq 1$,
\begin{align*}
2r |a_{2r,w_n}-a_{r,w_n}|&+ |b_{2r,w_n}-b_{r,w_n}|\\
&\leq C\left\| {w_n}/{x_N^{2s-1}} - Q_{2r} \right\|_{L^\infty(B_{2r}^+)}+C \left\| {w_n}/{x_N^{2s-1}} -Q_r \right\|_{L ^\infty(B_r^+)} \\
&\leq C \left( \th_n(2r) (2r)^{1+\a }+ \th_n(r) r^{1+\a}\right)\leq C \th_n(r) r^{1+\a }.
\end{align*}
Using the monotonicity  of  $\th_n$,  we see that for all $m\in \N$ such that $2^m r\leq 1$, we get 
\begin{align*}
&\| Q_{2^m r}-Q_r\|_{L ^\infty(B_{2^m r})}\leq C \| Q_{2^m r}-Q_r\|_{L ^2(B_{2^m r})} \leq  C\sum_{i=1}^m\| Q_{2^{i} r}-Q_{2^{i-1}}\|_{L ^2(B_{2^i r})}\\
&\leq C \sum_{i=1}^m 2^ir |a_{2^ir,w_n}-a_{2^{i-1}r,w_n}|+ |b_{2^ir,w_n}-b_{2^{i-1}r,w_n}|\leq C \th_n(r)\sum_{i=1}^m2^{i(1+\b )}  r^{1+\a}\\
&\leq C  \th_n(r) (2^mr)^{1+\a }.
\end{align*}
Hence for all $M\geq 1$ such that $Mr\leq 1$, we get 
$$
\| Q_{M r}-Q_r\|_{L ^\infty(B_{M r})}\leq  \th_n(r) (M r)^{1+\a }.
$$
From this,  \eqref{eq:def-th-n} and the monotonicity  of  $\th_n$, we then have, for $Mr\leq 1$,  
\begin{align*}
&\left\| {v_n}/{x_N^{2s-1}}  \right\|_{L^\infty(B_{M}^+)}=\frac{1}{r_n^{1+\a }\th_n(r_n)}\| {w_n}/{x_N^{2s-1}} -Q_{r_n} \|_{L^\infty(B_{M r_n }^+)}\\
&+\frac{1}{r_n^{1+\a }\th_n(r_n)}\| {w_n}/{x_N^{2s-1}} -Q_{M r_n} \|_{L^\infty(B_{Mr_n }^+)}\ +\frac{1}{r_n^{1+\a }\th_n(r_n)}\| Q_{M r_n}-Q_{r_n}\|_{L ^\infty(B_{M r_n})}\\
&\leq  C M^{1+\a }.
\end{align*}
As a consequence,
\be\label{eq:v_n-growth-1-Dir}
\|v_n\|_{L^\infty(B_M^+)}\leq  C M^{2s+\a} \qquad\textrm{ whenever $M r_n\leq 1$.}
\ee
Now  for $M r_n> 1$ and using  \eqref{eq:psi_n-c0},  we  have  (recall that   $2s+\a>1$)
$$
\|v_n\|_{L^\infty(B_M^+)}\leq \frac{C }{r_n^{2s+\a } \th_n(r_n)}(1+ r_n^{-\varrho}+ M r_n+|T_3^n(0)| M r_n)\leq  C M^{2s+\a+\varrho}.
$$
In view of this and \eqref{eq:v_n-growth-1-Dir}  we deduce that
\be\label{eq:v_n-growth-Dir}
\|v_n\|_{L^\infty(B_M^+)}\leq  C M^{2s+\b} \qquad\textrm{ for all $M \geq 1$}.
\ee
Letting  $\ti\Phi_n(x)=\frac{1}{r_n}\Phi(r_n x)$, we thus get 
\be\label{eq:final-eq-Dir} 
  \calB_{\ti \Phi_n}(\chi_Rv_n,\vp)=\int_{\R^N_+} G^n_R \vp\, dx\qquad\textrm{ for all $\vp\in C^\infty_c(B_{\frac{1}{r_n}}^+)$},
\ee
where 
$$
G^n_R(x):= \frac{1}{r_n^{\b}\th_n(r_n)} f_n(r_n x)+ c_{N,s} \int_{\R^N_+}\frac{(\chi_R(x)-\chi_R(y)) v_n(y)}{|\ti\Phi_n(x)-\ti\Phi_n(y)|^{N+2s}} \, dy  - \frac{1 }{r_n^{\b}\th_n(r_n)} (  \cL_{\Phi_n}   P_n)(r_n x),
$$
$$
P_n(x)=  \ov\eta(x_N) \left( a_{r_n,w_n} \cdot x' +b_{r_n, {w_n}} (1-T_3^n(0)x_N) \right) x_N^{2s-1}
$$
and  $ (\cL_{\Phi_n}   P_n)(x)=c_{N,s} \int_{\R^N_+}\frac{P_n(x)-P_n(y) }{|\Phi_n(x)-\Phi_n(y)|^{N+2s}} \, dy$.
By Lemma \ref{lem:Dir-Op-bdr},  
\be\label{eq:h1h2}
 (\cL_{\Phi_n}   P_n )(r_n x)=g_{1,n}(r_n x)\log(r_n x_N)+ g_{2,n}(r_n x)
\ee
and thus using \eqref{eq:ap-est} we see that 
\begin{align}\label{eq:cL-Pn}
\frac{1 }{r_n^{\b}\th_n (r_n)}|( \cL_{\Phi_n}   P_n)(r_n x)-( g_{2,n}(r_n x)- g_{2,n}(0)) |&\leq \frac{C r_n^\b |x|^\b}{r_n^{\a}\th_n(r_n)}\left( 1+|\log x_N|+|\log r_n|\right) (r_n^{-\varrho}+1) \nonumber\\
&\leq \frac{C|x|^\b}{\th_n(r_n)}\left(1+ |\log x_N|\right).
\end{align}
Since $|\n \phi_n|\leq \frac{1}{4}$ and $\ti \Phi_n(x)=x+(0,r_n^{-1}\phi_n(r_n x))$,    we obtain
\begin{align} \label{eq:bi-Li-Dir}
2|x-y|\geq  | \ti\Phi_n(x)-\ti\Phi_n(y) |\geq \frac{1}{2}|x-y|, \qquad |D\ti\Phi_n(x)| \leq 2  \qquad\textrm{ for all $x,y\in \R^N$. }
\end{align}
  Define, for $x\in B_{  R/2}$, 
\begin{align*}
F_R^n(x)&:=c_{N,s}\int_{\R^N_+}\frac{(\chi_R(x)-\chi_R(y)) v_n(y)}{|\ti\Phi_n(x)-\ti\Phi_n(y)|^{N+2s}}  \, dy\\
&= c_{N,s}\int_{\R^N_+\cap \{|y|\geq R\}}\frac{(1-\chi_R(y)) v_n(y)}{|\ti\Phi_n(x)-\ti\Phi_n(y)|^{N+2s}}  \, dy.
\end{align*}
By the mean value theorem and the fact that $\ti\Phi_n(0)=0$, we have  $F_R^n(x)=F_R^n(0)+\cR_R^n(x)$,  where 
$$
 \cR_R^n(x):=-(N+2s)c_{N,s}\int_0^1 \int_{\R^N_+ }\frac{(\ti\Phi_n(x)-\ti\Phi_n(y))  \cdot \ti\Phi_n(x)}{|\ti\Phi_n(x)-\ti\Phi_n(y)|} \frac{(1-\chi_R(y)) v_n(y)}{| t \ti\Phi_n(x)-\ti\Phi_n(y)|^{N+2s+1}}  \, dy\, dt.
$$
Note that by \eqref{eq:bi-Li-Dir},   $| t \ti\Phi_n(x)-\ti\Phi_n(y)|\geq \frac{1}{4} |y|$ whenever  $|y|\geq R$ and $|x|\leq \frac{R}{4}$.
Therefore, 
by \eqref{eq:v_n-growth} and \eqref{eq:bi-Li-Dir},    there exists $C>0$ such that
\be\label{eq:cRnR-Dir}
 | \n \cR_R^n(x) |\leq {C}  R^{\b-1} \qquad\textrm{ for all $R>1$,  $|x|\leq \frac{R}{4}$ and $n\geq 2$}
\ee 
and  by construction we have  $\cR_R^n(0)=0$.
We can thus write
\be\label{def-GnR-Dir}
G^n_R(x)= H^n_R(x)+ c^n_R+\cR^n_R(x), 
\ee
with  (recalling \eqref{eq:h1h2})
$$
H^n_R(x)= \frac{1}{r_n^{\a}\th_n(r_n)}\left\{  f_n(r_n x)- f_n(0)-(\cL_{\Phi_n} P_n)(r_nx)+( g_{2,n}(r_n x)- g_{2,n}(0)) \right\}  
$$
and 
$$
c^n_R= \frac{1}{r_n^{\a}\th_n(r_n)}\left(    f_n(0)- g_{2,n}(0)  \right)  + F_R^n(0).
$$
Recalling that $\|f_n\|_{C^\a(\R^N_+)}\leq 1$,   we then get from \eqref{eq:cL-Pn}  that 
$$
|H^n_R(x)|\leq   \frac{C( R)} {\th_n(r_n)} (1+|\log x_N|)\qquad\textrm{ for all $x\in B_{R/2}^+$,} 
$$
so that 
\be\label{def-HnR-Dir}
\|H^n_R\|_{L^p(B_{R/2})}\leq   \frac{C( R)} {\th_n(r_n)}  \qquad\textrm{ for all $p>1$.} 
\ee
Let us now show  that $(c_R^n)_{n\in \N}$ is bounded.  For this, we pick $\phi\in C^\infty_c(  B_{1/4}^+)$ with $\int_{\R^N}\phi\, dx=1$ and let $R\geq2$.
  Then  multiply \eqref{eq:final-eq-Dir} by $\phi$ and integrate over $\R^N_+$ to get
\begin{align}\label{eq:cnR-Dir}
|c^n_R|\leq \int_{ \R^N}| \chi_R(x) v_n(x)| |\cL_{\ti\Phi_n}\phi(x)|\, dx+  \int_{ B_{R/2}}| H^n_R(x) \phi(x)|dx+ \int_{ B_{R/2}}| \cR^n_R(x) \phi(x)|dx\leq C(R).
\end{align}

In view of this,  \eqref{eq:final-eq-Dir},  \eqref{def-GnR-Dir}, \eqref{def-HnR-Dir} and  \eqref{eq:cRnR-Dir}, we can apply Lemma \ref{lemm:appr-esrtim-Dir}, to deduce that
$v_n$ is bounded in $C^{2s-1}_{loc}(\ov {\R^N_+})\cap H^{s}_{loc}(\ov {\R^N_+})$ and converges to some $v\in C^{2s-1}_{loc}(\ov {\R^N_+})\cap H^{s}_{loc}(\ov {\R^N_+})$.   Moreover  $v_n/x_N^{2s-1}$ converges to $v/x_N^{2s-1}$  in $C^{1-\e}_{loc}(\ov {\R^N_+})$,  for all $\e\in (0,1)$.
In addition passing to the limit in \eqref{eq:to-contradict-Dir}, \eqref{eq:v_n-growth-Dir} and \eqref{eq:perp-Dir}, we get 
\be\label{eq:v-growth}
  |v(x)|\leq C  |x|^{2s+\b} \qquad\textrm{ for all $x\in \R^N$,}  
\ee
\be\label{eq:v-contradiction-Dir}
\|v/x^{2s-1}_N\|_{L^\infty(B_1^+)}\geq \frac{1}{2}, \qquad   
   \int_{B_{1}^+} \frac{v}{x_N^{2s-1}}(1-T_3^\infty(0)x_N) \, dx= \int_{B_1^+} \frac{v}{x_N^{2s-1}} x_i\, dx=0,   
\ee
 { for $i=1,\dots,N-1$,} where $T_3^\infty(0)$ is the  limit   $T_3^n(0)$  as $n\to \infty$ and satisfies $|T_3^\infty(0)|\leq 1/4$. 
Passing to the limit in \eqref{eq:final-eq-Dir},  \eqref{eq:cRnR-Dir} and using \eqref{eq:cnR-Dir},  we get,  for all $\phi\in C^\infty_c(B_{\frac{R }{4}})$,
$$
\int_{\R^{2N}_+}\frac{((\eta_R v)(x)-(\eta_Rv)(y))(\vp(x)-\vp(y))}{|x-y|^{N+2s}}  dydx=\int_{\R^N_+}(c_R^\infty+\cR^\infty_R(x))\vp(x)\, dx
$$
  and 
\be \label{eq:cR-cR}
c^\infty_R\in \R, \qquad   \cR_R^\infty(0) =0, \qquad | \n \cR_R^\infty(x) |\leq {C}  R^{\b-1} \qquad\textrm{ for all $R\geq 2$ and  $x\in B_{\frac{R }{2}}$.}
\ee
Consequently,  since $\b<\a_s<1$, 
$$
\Ds_{{\R^N_+}} v \stackrel{\b}{=}0 \qquad\textrm{ in $\R^N_+$}.
$$
Applying  Theorem \ref{lem:from-ND-1D} and using \eqref{eq:v-growth},  we deduce that  $\frac{v}{x_N^{2s-1}}(x)=b+ a\cdot x'$. This is in contradiction with \eqref{eq:v-contradiction-Dir}. 
\end{proof} 

\begin{corollary}\label{cor:near-bdr}
Under the assumptions of Proposition \ref{prop-blow-u-Dir},  there exist    $a\in \R^{N-1} $ and $b\in \R$ with  $|a|+|b|\leq C$ such that letting $\psi=\frac{w}{x_N^{2s-1}},$ we get  
\be\label{eq:dec-zero-ok}
 \left\|\psi-  \left( (a\cdot x'+b)- b T_3(0)  x_N\right)  \right\|_{L^\infty(B_r^+)}\leq C r^{1+\a } \qquad\textrm{ for all $r>0$.}
\ee
Moreover $a=\n_{x'} \psi (0)$,  $ b=\psi (0)$ and $\de_{x_N} \psi(0)=- b T_3(0)$.
\end{corollary}
\begin{proof}
Let $
Q_r(x)=  a_{r,w} \cdot x' +b_{r, {w}} (1-T_3(0)x_N) .
$
Let   $0<\rho_2\leq  \rho_1/4\leq 1/4 $.  Pick $k\in \N$ and   $\s\in [1/4,1/2]$ such that $\rho_2=\s^k\rho_1$.  Then provided $r\in (0,1)$,  by \eqref{eq:Q2rQr} and \eqref{eq:to-contra-bdr}  we get 
\begin{align*}
&\rho_1 |a_{\rho_2,w}-a_{\rho_1,w}|+ |b_{\rho_2,w}-b_{\rho_1,w}|\\
&\leq \sum_{i=0}^{k-1}\s^{i+1} \rho_1 |a_{\s^{i+1}\rho_1,w}- a_{\s^{i}\rho_1,w}|+\sum_{i=0}^{k-1}  |b_{\s^{i+1} \rho_1,w}-b_{\s^{i}\rho_1,w}|\\
& \leq C\frac{1}{|B_{\s^{i+1}\rho_1}|^{\frac{1}{2}}}\| Q_{\s^{i+1}\rho_1}-Q_{\s^ i \rho_1} \|_{L^2(B_{\s^{i+1} \rho_1})}\\
&\leq C   \sum_{i=0}^{k-1}  \left(  \|\psi -Q_{\s^i\rho_1 }  \|_{L^\infty\left(B_{\s^i\rho_1}\right)}+ \|\psi- Q_{\s^{i+1}\rho_1 }\|_{L^\infty\left(B_{\s^{i+1}\rho_1}\right)}  \right) \\
&\leq C  \sum_{i=0}^{k-1}  \left( (\s^{i}\rho_1)^{ 1+\a} +(\s^{i+1}\rho_1)^{ 1+\a}      \right) \leq C  \rho_1^{1+\a}  .
\end{align*}
This implies that the limits $a:=\lim_{r\to 0} a_{r,w}$ and $b:=\lim_{r\to 0} b_{r,w}$ exist.  Moreover, letting $\rho_2\to 0$ and $\rho_1=1/2$ in the above estimate,  we get  $|a|+|b|\leq C$.   Hence  \eqref{eq:dec-zero-ok} follows. \\
It is by now easy to see from \eqref{eq:dec-zero-ok},   that $b=\psi(0)$ and $\n \psi(0) =(a,-bT_3(0))$.
\end{proof}
\begin{corollary}\label{cor:u-odr-ds}
Under the hypothesis of Proposition \ref{prop-blow-u-Dir}, there exists a constant $C$ depending only on $N,s,\a$ and $\b $ such that 
for all ${\x\in B_{1}^+\cap \R  e_N} $, 
$$
  [ w/x_N^{2s-1}]_{C^{1+\a}({B_{\x_N/2}(\x)})}\leq  C .
$$

\end{corollary}
\begin{proof}
Let $\x:=(0,2\rho ) \in B_{1}^+$ with $\rho<1/2$.  We define
$$
P(x)=\ov \eta(x_N) (x_N)^{2s-1}_+ (a\cdot x'+b)- b T_3(0) \ov \eta(x_N) (x_N)^{2s}_+  ,
$$ 
where $b $ and $a$ are given by Corollary \ref{cor:near-bdr} and $T_3(0)$ is given by Lemma \ref{lem:Dir-Op-bdr}. Put 
 $ v_\x(y)=w(\x+\rho y) -P(\x+\rho y)  $,  $f_\x(y)= \rho^{2s}f(\x+\rho y)$.   We have 
\be \label{eq:Dir-near-final-J}
\calL_{\Phi_\rho} v_\x=      f_\x( x)+ \rho^{2s}(\calL_{\Phi} P)(\x+\rho x)  \qquad\textrm{  for all $\vp \in C^\infty_c(B_1^+)$,}
\ee
where  $\Phi_{\rho}(x)=\rho^{-1} \Phi(\rho x+\x)=(x',x_N+ \rho^{-1}\phi(\rho x'))$ and 
$$
\calL_{\Phi_\rho } v (x)=c_{N,s}\int_{\{y_N>-2\}} \frac{v(x)-v(y) }{|\Phi_\rho (x)-\Phi_\rho (y)|^{N+2s}}\, dy.
$$
By Lemma \ref{lem:Dir-Op-bdr}, we have 
\begin{align*}
F(x):=(\calL_{\Phi} P)(\x+\rho x)&=g_1(\x+\rho x)\log(2\rho+\rho x_N)+ g_2(\x+\rho x)
\end{align*}
from which it  follows  that
$\|F\|_{C^\b([0,1))}\leq C  \rho^\b|\log \rho|\leq C \rho^\a$.  Note that $|u_\x(x)|\leq C(1+ |x|^{2s+\a})$ and thus since $\Phi_\rho\in C^2(\R^N;\R^N)$ is a global (volume preserving) diffeomorphism,   by a change of variable and \eqref{eq:Dir-near-final-J}, we obtain
$$
\Ds_{\O_\rho}  v_\x\circ\Phi_\rho^{-1}  {=}  G_\x\circ\Phi_\rho^{-1} \qquad    \textrm{  in $\O_\rho$,}
$$
where  $G_\x(x):= f_\x( x)+ \rho^{2s} F(\x+\rho x)$ and $\O_\rho=\Phi_\rho^{-1} (\{y_N>-2\})$.  We note that
  $[G_\x]_{C^\a(B_{1/2})}\leq \rho^{2s+\a}$,  $\| v_\x\circ\Phi_\rho^{-1}\|_{L^\infty(B_{1/2})}\leq \rho^{2s+\a} $ and $B_1\subset\O_\rho$.  We then apply   Lemma \ref{lem:int-Scuader}  to deduce that 
\begin{align*}
\|v_\x\circ\Phi_\rho^{-1} \|_{C^{2s+\a}(B_{1/4})}&\leq C \left( \int_{\O_\rho\setminus B_{1/8}}|y|^{-N-2s-1}|v_\x\circ\Phi_\rho^{-1} (y)|dy+  \rho^{2s+\a}   \right)\\
& \leq C   \left( \int_{\R^N_+\setminus B_{1/16}}|x|^{-N-2s-1}|v_\x (x)|dx+  \rho^{2s+\a}   \right)\\
&\leq C   \left( \sum_{i=0}^\infty\int_{2^{i-3}>|x|>2^{i-4}}|x|^{-N-2s-1}|v_\x (x)|dx+  \rho^{2s+\a}   \right)\\
&\leq C   \left( \sum_{i=0}^\infty (2^i\rho)^{-N-2s-1} \rho^{2s+1} \int_{2^{i-3}\rho> |z|>2^{i-4}\rho} v_\x (z/\rho)|dz+  \rho^{2s+\a}   \right)\\
&\leq C   \left( \sum_{i=0}^\infty (2^i\rho)^{-N-2s-1} \rho^{2s+1}  \int_{2^{i}\rho> |x| }|w(x)-P(x)|dx+  \rho^{2s+\a}   \right)\\
&\leq C\left( \sum_{i=0}^\infty 2^{i{(\a-1)}}\rho^{2s+\a}+  \rho^{2s+\a}   \right)\leq C\rho^{2s+\a}  ,
\end{align*}
where we used \eqref{eq:dec-zero-ok}  which yields $|w(x)-P(x)|\leq C |x|^{2s+\a}$.  We conclude that
$$
\|v_\x \|_{C^{2s+\a}(B_{1/4})} \leq C\rho^{2s+\a} .
$$
Scaling and translating  back,  we then get 
$$
\| w-P\|_{L^\infty(B_{\rho/4}(\x))}\leq C\rho^{2s+\a},  \qquad  [ \n w-\n P]_{C^\a(B_{\rho/4}(\x))}\leq C\rho^{2s-1}.
$$
Combining the above estimates with 
 $$
 \|1/x_N^{2s-1}\|_{L^{\infty}(B_{\rho/4}(\x) )}\leq C\rho^{1-2s}, \qquad [\n( 1/x_N^{2s-1})]_{C^{\a}(B_{\rho/4}(\x) )}\leq C\rho^{-2s-\a}
 $$
  we finally get
$$
[\n (w/x^{2s-1})-\n (P/x^{2s-1})]_{C^{\a}(B_{\rho/4}(\x))}\leq C.
$$
Since $[ \n (P/x^{2s-1})]_{C^{\a}(B_{\rho/4}(\x))}=0$,  this yields 
$$
[\n (w/x^{2s-1})]_{C^{\a}(B_{\rho/4}(\x))}\leq C.
$$
and the proof is complete.
\end{proof}

We now prove the following.

\begin{theorem}\label{th-2}
Given $s\in(\frac12,1)$, let $\alpha_s\in(0,1-s)$ be given by Lemma~\ref{analysis-beta}.
Let $\Omega\subset \R^N$ be a $C^{2,\b}$ domain, with $0\in \partial\Omega$, and $f\in C^\a(\overline\Omega\cap B_1)$ with $\a<\min\{\alpha_s,2s-1,\b\}$ and $\a+2s\neq1$.
Let $u\in H^s_0(\O)$ be such that 
$$
 D_\O(u,\vp)=\int_{\O}f\vp\, dx \qquad \textrm{$\forall\vp \in C^\infty_c(B_1\cap\O)$,}
$$
and let $\d=\textrm{dist}(x,\de\O)$.  
Then,
\be\label{eq:Schauder-ud} 
\|u/\d^{2s-1}\|_{C^{1+\a}(B_{\rho/2}\cap\ov \O)}\leq C\left(\|u\|_{L ^2(\O)}+ \|f\|_{C^\a(\ov \O)}  \right),
\ee
with $C,\rho>0$ depending only on $N$, $s$, $\a$, and $\Omega$.

Moreover, letting $\psi=u/\d^{2s-1}$, we have 
\be\label{eq:unn-mc}
\de_\nu \psi (\s)=-(N-1) H_{\de\O}(\s) \psi(\s) \qquad\textrm{ for all $\s\in \de\O\cap B_1$,}
\ee
where    $H_{\de\O}$ is the mean curvature of $\de\O$.
\end{theorem}

\begin{proof}
We assume for simplicity that $\|u\|_{L^2(\O)}+\|f\|_{C^{\a}(\ov \O)}\leq 1.$  Hence by Lemma \ref{lemm:appr-esrtim-Dir}, we get $\|u/\d^{2s-1}\|_{L^\infty(\O\cap B_{1/2})}+\|u\|_{C^{2s-1}(\ov \O\cap B_{1/2})}\leq C$.

 Up to  a scaling and  a rotation,  we may assume that $B_{1/4}\cap \de\O$ is contained in  the graph $\{x_N=\g(x')\}$ with   $\g\in {C^{2,\b}(B_{1/4}')}$.
  For $q\in \de\O $, let $\nu(q)$ denote the unit exterior normal vector of $\O$ at $q$ and denote $T_q\de\O=\nu(q)^\perp$ the tangent plane of $\de\O$ at $q$. \\
  By the local inversion theorem,  there exist $r_0\in (0,1/4]$ depending only on $N$, $\b$ and $\|\g\|_{{C^{2,\b}(B_1')}  }$ such that for any $q\in B_{r_0}\cap \de\O$,  there exist an orthonormal basis $(E_1(q)\dots,E_{N-1}(q))$  of $T_q\de\O$ and   $ \ti \phi\in C^{2,\b}(  B_{r_0}' )$ such that    
$$
x'\mapsto q+\sum_{i=1}^{N-1} x_i E_i(q)- \ti \phi(x')\nu(q):  B_{r_0}' \to \de\O
$$
with $\|\ti\phi\|_{C^{2,\b}(  B_{r_0}' )}\leq C(N,\b, \g,r_0)$.  It is clear that $ \ti \phi(0)=|\n\ti \phi(0)|=0$.
 We let $\eta\in C^\infty_c(B_{r_0}')$ with $\eta\equiv 1$ on $B_{r_0/2}'$ and put $\phi=\eta \ti\phi$.
For all $q\in B_{r_0}\cap \de\O$, we define
\be \label{eq:Phi-q}
\Phi:\R^N\to \R^N, \qquad \Phi(x',x_N)=q+ \sum_{i=1}^{N-1} x_i E_i(q)-\left(x_N+ \phi(x') \right)\nu(q).
\ee
Next, decreasing $r_0$ if necessary,    $\Phi$  is a global $C^{2,\b}$ diffeomorphism  and clearly  $\textrm{Det}D\Phi(x)=1$ for all $x\in \R^N$.   We may assume that the (signed) distance function $d\in C^{2,\b}(\{d<r_0/2\})$.  
We also recall that the mean curvature of $\de\O$ at $q$ is given by
 \be\label{eq:def-MC}
 H_{\de\O}(q):=\frac{1}{N-1}\D\phi(0).
\ee
We define $Q_\d:=B_\d'\times (0,\d)$ and  $\calQ_\d:=\Phi(Q_\d )\subset \O$.   
  For any   $x_0\in \calQ_{r_0/9}$, let  $q\in \de\O\cap B_{r_0}$   such that $|x_0-q|=d(x_0)$.
Letting $v={ \chi_{r_0/4}}u$,  we have 
\begin{align*}
\Ds_{\O} v=f+c_{N,s}\int_{\O}\frac{(1-\chi_{ {r_0/4}}(y))u(y)}{|x-y|^{N+2s}}\, dy=:f_1.
\end{align*}
Clearly $\|f_1\|_{C^\a(\calQ_{r_0/8})}\leq C$ for all $\a\in (0, \min (2s-1,\b))$.
Now we get
\begin{align*}
\Ds_{\calQ_{ r_0/2}} v= f_1- \Ds_{\O\setminus \calQ_{ r_0/2}} v=: f_2, 
\end{align*}
and obviously $\|f_2\|_{C^\a(\calQ_{r_0/8})}\leq C$ because $ \|v\|_{C^{2s-1}(\calQ_{r_0/8})}\leq C$.  We define $ w(x)=v(\Phi(x))$ and $g_2(x)=f_2(\Phi(x))$.  Then by a change of variable, we get 
\begin{align*}
c_{N,s} \int_{Q_{r_0/2}}\frac{ w(x)- w(y)}{|\Phi(x)-\Phi(y)|^{N+2s} }\, dy=g_2(x) .
\end{align*}
Hence $w\in H^s(\R^N_+)$ satisfies $\|w/x^{2s-1}_N\|_{L^{\infty}(\R^N_+)}+\|w\|_{C^{2s-1}(\R^N_+)}\leq C$ and  solves
\begin{align}\label{eq:Lw-g3}
c_{N,s}\int_{\R^N_+}\frac{ w(x)- w(y)}{|\Phi(x)-\Phi(y)|^{N+2s} }\, dy=g_3(x),
\end{align}
with $g_3(x)=g_2(x)-c_{N,s}\int_{\R^N_+\setminus Q_{r_0/2}}\frac{ w(x)- w(y)}{|\Phi(x)-\Phi(y)|^{N+2s} }\, dy.$ We have that $\|g_3\|_{C^\a(Q_{r_0/8} )}\leq C.$ 

Recall that $x_0=q-|x_0- q|\nu(q)=\Phi(0,  |x_0- q|).$
We can thus apply Corollary \ref{cor:u-odr-ds} to the equation \eqref{eq:Lw-g3},  to get
\be \label{eq:wx2s-1}
  [ w/x_N^{2s-1}]_{C^{1+\a}(B_{r/2}(\x))}=  [ u\circ\Phi/x_N^{2s-1}]_{C^{1+\a}({B_{r/2}(\x)})}\leq  C,  \quad\textrm{ with $\x=\Phi^{-1 }(x_0)  \in Q_{\frac{r_0}{9}} \cap \R e_N$.}
\ee
where $r=\x_N/2= |x_0- q|/2$.
In addition, Corollary \ref{eq:dec-zero-ok},  \eqref{eq:C30-MC}  and  \eqref{eq:def-MC} yield
\be \label{eq:dex_N-mc}
\de_{x_N} [w/x_N^{2s-1}] (0)=-T_3(0)   [ w/x_N^{2s-1}](0)= (N-1) H_{\de\O}(q)  [ w/x_N^{2s-1}](0).
\ee
In view of \eqref{eq:Phi-q} and \eqref{eq:wx2s-1},  by changing variables,  we  thus obtained  
$$
  [ u/ \d^{2s-1}]_{C^{1+\a}({B_{r/8}(x_0)})}\leq  C ,  
$$
where   we used that $x\mapsto \frac{d(\Phi(x))}{x_N}\in C^{1+\a}(B_{r_0/2})$ and that $d=\d$ in $\O$.
This implies that  $[ u/\d^{2s-1}]_{C^{1+\a}(\calQ_{r_0/9})}\leq  C  $ and \eqref{eq:Schauder-ud}  follows.

  Finally  using \eqref{eq:dex_N-mc} and the fact that $d(\Phi(0,x_N))=d(0, x_N)=x_N+O(x_N^{2})$, we conclude that
$$
\de_{\nu}[u/\d^{2s-1}] (q)= - (N-1) H_{\de\O}(q) [u/\d^{2s-1}] (q)
$$
and the proof is complete.
\end{proof}

We finally give the:

\begin{proof}[Proof of Theorem \ref{thm-Dir}]
The result follows from Theorem \ref{th-2} and a covering argument.
\end{proof}

\end{document}